\documentclass[10pt]{amsart}
\usepackage{amsmath}
\usepackage{amssymb}
\usepackage{enumerate}
\usepackage{amsbsy}
\usepackage{amsfonts}
\usepackage{color}
\usepackage{dsfont,amsxtra,latexsym,amscd,amsthm}
\usepackage{tikz,graphicx}
\usetikzlibrary{arrows,shapes}

\newtheorem{thm}{Theorem}[section]
\newtheorem{prop}[thm]{Proposition}
\newtheorem{lem}[thm]{Lemma}

\newtheorem{cor}[thm]{Corollary}

\numberwithin{equation}{section}

\theoremstyle{remark}
\newtheorem{rem}{Remark}

\newcommand{\supp}{\text{supp }}

\newcommand{\R}{\mathbb{R}}

\newcommand{\Z}{\mathbb{Z}}
\newcommand{\ep}{\varepsilon}

\newcommand{\ls}{\lesssim}

\newcommand{\la}{\langle}
\newcommand{\ra}{\rangle}

\newcommand{\tph}{\mathbb T_h^d}

\begin{document}
\title[]
{Uniform Strichartz estimates on the lattice}

\author{Younghun Hong}
\address{Department of Mathematics, Chung-Ang University, Seoul 06974, Republic of Korea}
\email{yhhong@cau.ac.kr}

\author{Changhun Yang}
\address{Department of Mathematical Sciences, Seoul National University, Seoul 151-747, Republic of Korea}
\email{maticionych@snu.ac.kr}

\subjclass[2010]{}

\maketitle
\begin{abstract}
In this paper, we investigate Strichartz estimates for discrete linear Schr\"odinger and discrete linear Klein-Gordon equations on a lattice $h\mathbb{Z}^d$ with $h>0$, where $h$ is the distance between two adjacent lattice points. As for fixed $h>0$, Strichartz estimates for discrete Schr\"odinger and one-dimensional discrete Klein-Gordon equations are established by Stefanov-Kevrekidis \cite{SK2005}. Our main result shows that such inequalities hold uniformly in $h\in(0,1]$ with additional fractional derivatives on the right hand side. As an application, we obtain local well-posedness of a discrete nonlinear Schr\"odinger equation with a priori bounds independent of $h$. The theorems and the harmonic analysis tools developed in this paper would be useful in the study of the continuum limit $h\to 0$ for discrete models, including our forthcoming work \cite{HY} where strong convergence for a discrete nonlinear Schr\"odinger equation is addressed.
\end{abstract}

\section{Introduction}

We consider a discrete linear Schr\"odinger equation 
\begin{equation}\label{eq: discrete schrodinger}
\left\{\begin{aligned}
i\partial_tu+\Delta_h u&=0,\\
u(0)&=u_0, 
\end{aligned}\right.
\end{equation}
and a discrete linear Klein-Gordon equation
\begin{equation}\label{eq: discrete KG}
\left\{\begin{aligned}
\partial_t^2u-\Delta_h u+ m^2u&=0,\\
(u(0),\partial_tu(0))&=(u_0,u_1)
\end{aligned}\right.
\end{equation}
on a lattice domain 
$$\mathbb{Z}_h^d:=h\mathbb{Z}^d=\left\{x=hn: n\in\mathbb{Z}^d\right\}\textup{ with }h>0,$$
where
$$u=u(t,x):\mathbb{R}\times\mathbb{Z}_h^d\to\mathbb{C}.$$
We here define the discrete Laplacian by 
$$(\Delta_h u)(x)=\sum_{j=1}^d\frac{u(x+he_j)+u(x-h e_j)-2u(x)}{h^2},$$
where $\{e_j\}_{j=1}^d$ is the standard basis. In other words, we consider the harmonic oscillators interacting only with their nearest neighbors.

Discrete Schr\"odinger and discrete Klein-Gordon equations have been extensively studied in various aspects in the physics literature. 
Discrete Schr\"odinger equations describe periodic optical structures created by coupled identical single-mode linear waveguides \cite{eisenberg2002optical,peschel2002optical,sukhorukov2003spatial}. They are also closely related to nonlinear dynamics of the Bose-Einstein condensates in optical lattices \cite{cataliotti2001josephson,cataliotti2003superfluid}.
Meanwhile, discrete Klein-Gordon equations describe Fluxon dynamics in one dimension parallel array of Josephson Junctions \cite{ustinov1993fluxon}, and also arises as a model for local denaturation of DNA \cite{peyrard1989statistical}.
In \cite{Mingaleev1999}, the equations of motion of the model of DNA dynamics are reduced to the nonlocal discrete nonlinear Schr\"odinger equations, which shown rigorously to converges to fractional Schr\"odinger equations in continuum limit by \cite{KLetal-13}.
For more informations on survey and general theory of discrete equations, see \cite{kevrekidis2009discrete,kevrekidis2001discrete}.
Our focus is particularly on developing analytic tools to explore the continuum limits $h\to 0$ of the above equations. Precisely, we aim to establish inequalities quantitatively measuring decay properties of solutions, namely Strichartz estimates, but in the meantime, we also want them to hold uniformly in $h\in(0,1]$.

For $1\leq p<\infty$ (respectively, $p=\infty$), the function space $L_h^p$ consists of all complex-valued functions on $\mathbb{Z}_h^d$ satisfying 
$$\sum_{x\in \mathbb{Z}_h^d} |f(x)|^p<\infty\quad\left(\textup{respectively, }\sup_{x\in \mathbb{Z}_h^d} |f(x)|<\infty\right).$$
If $f\in L_h^p$, then its $L_h^p$-norm is defined by 
\begin{equation}\label{L^p norm}
\|f\|_{L_h^p}:=h^{\frac{d}{p}}\|f\|_{\ell_x^p(\mathbb{Z}_h^d)}=\left\{\begin{aligned}
&\left\{h^d\sum_{x\in \mathbb{Z}_h^d} |f(x)|^p\right\}^{1/p}&&\textup{if}&&1\leq p<\infty;\\
&\sup_{x\in \mathbb{Z}_h^d} |f(x)|&&\textup{if}&&p=\infty.
\end{aligned}\right.
\end{equation}
Putting the constant $h^{\frac{d}{p}}$ in the norm $\|\cdot\|_{L_h^p}$ is natural in consideration of the continuum limit $h\to 0$, since for $f\in L^p(\mathbb{R}^d)$,
$$\int_{\mathbb{R}^d} |f(x)|^p dx\approx h^d\sum_{x\in \mathbb{Z}_h^d} |f_h(x)|^p\quad\textup{as}\quad h\to 0,$$
where $f_h(x)$ denotes the average of $f$ on the $h$-cube centered at $x\in\mathbb{Z}_h^d$, i.e., 
$$f_h(x):=\frac{1}{h^d}\int_{x+h[-\frac{1}{2},\frac{1}{2}]^d}f(y) dy.$$
By the above definitions, the previously-known dispersion and Strichartz estimates for the discrete Schr\"odinger equation \eqref{eq: discrete schrodinger} are written as follows.
\begin{thm}[Stefanov-Kevrekidis \cite{SK2005}]
	$(i)$ (dispersion estimate)
	\begin{equation}\label{SK dispersion estimate}
	\|e^{it\Delta_h}u_0\|_{L_h^\infty}\leq\frac{C}{|th|^{d/3}}\|u_0\|_{L_h^1}.
	\end{equation}
	$(ii)$ (Strichartz estimates) We say that $(q,r)$ is discrete Schr\"odinger-admissible if 
	\begin{equation}\label{S-admissible}
	\frac3q+\frac {d}{r}=\frac{d}{2}, \quad 2\leq q,r\leq\infty,\quad (q,r,d)\neq (2,\infty,3).
	\end{equation}
	For any admissible pairs $(q,r)$ and $(\tilde{q},\tilde{r})$, we have
	\begin{equation}\label{SK Strichartz1}
	\| e^{it\Delta_h} u_0\|_{L_t^q(\mathbb{R};L_h^r)}\leq \frac{C}{h^{1/q}} \|u_0\|_{L_h^2}
	\end{equation}
	and
	\begin{equation}\label{SK Strichartz2}
	\left\|\int_0^t e^{i(t-s)\Delta_h}F(s)ds\right\|_{L_t^q(\mathbb{R};L_h^r)}\leq \frac{C}{h^{\frac{1}{q}+\frac{1}{\tilde{q}}}} \|F\|_{L_t^{\tilde{q}'}(\mathbb{R};L_h^{\tilde{r}'})}.
	\end{equation}	
\end{thm}

\begin{rem}
	$(i)$ The $|t|^{-d/3}$-decay in the dispersion estimate \eqref{SK dispersion estimate} is weaker than that for the continuum equation 
	$$\| e^{it\Delta_{\mathbb{R}^d}} u_0\|_{L^\infty(\mathbb{R}^d)}\leq \frac{C}{|t|^{d/2}} \|u_0\|_{L^1(\mathbb{R}^d)}$$
	due to the lattice resonance. Indeed, a solution to a discrete Schr\"odinger equation can be written as a certain oscillatory integral (see \eqref{oscillatory integral}), but its phase function may have degenerate Hessian. Thus, it only allows a weaker dispersion estimate and Strichartz estimates with different admissibility conditions.\\
	$(ii)$ The inequalities \eqref{SK dispersion estimate}, \eqref{SK Strichartz1} and \eqref{SK Strichartz2} cannot be directly applied to the continuum limit  problems, because the constants blow up as $h\to 0$ except the trivial case $q=\tilde{q}=\infty$.
\end{rem}

The main observation of this article is that the $h$-dependence in \eqref{SK Strichartz1} and \eqref{SK Strichartz2} can be removed paying fractional derivatives on the right hand side, which compensates the lattice resonance. We also prove that putting such additional derivatives is necessary for uniform boundedness.

As for a fractional derivative, we here adopt the definition as the Fourier multiplier of symbol $|\xi|^s$, and we use the homogeneous and the inhomogeneous Sobolev norms defined by 
\begin{equation}\label{Sobolev norm}
\begin{aligned}
\|f\|_{\dot{W}_h^{s,p}}:&=\||\nabla_h|^sf\|_{L_h^p}=\left\|(|\xi|^s\hat{f})^\vee\right\|_{L_h^p},\\
\|f\|_{W_h^{s,p}}:&=\|\langle\nabla_h\rangle^sf\|_{L_h^p}=\left\|(\langle\xi\rangle^s\hat{f})^\vee\right\|_{L_h^p},
\end{aligned}
\end{equation}
where $\hat{\cdot}$ (respectively, $\check{\cdot}$) is the lattice Fourier (respectively, inverse Fourier) transform on $\mathbb{Z}_h^d$ (see Section 2). In particular, we denote $\dot{H}_h^s:=\dot{W}_h^{s,2}$ and $H_h^s:=W_h^{s,2}$. Indeed, there are several alternative ways to define Sobolev norms in a discrete setting, but they are all equivalent. 

\begin{prop}[Norm equivalence]\label{norm equivalence}
	For any $1<p<\infty$, we have
	$$\|f\|_{\dot{W}_h^{s,p}}\sim \|(-\Delta_h)^{\frac{s}{2}}f\|_{L_h^p}\quad\forall s\in\mathbb{R}$$
	and
	$$\|f\|_{\dot{W}_h^{1,p}}\sim \sum_{j=1}^d\|D_{j;h}^+f\|_{L_h^p},$$
	where
	$$D_{j;h}^+f(x):=\frac{f(x+he_j)-f(x)}{h}.$$
\end{prop}

Using the Sobolev norm \eqref{Sobolev norm}, our main theorem is stated as follows. 

\begin{thm}[Uniform Strichartz estimates for a discrete Schr\"odinger equation]\label{thm: Strichartz for Schrodinger}
	Let $h\in (0,1]$. For any discrete Schr\"odinger-admissible pairs $(q,r)$ and $(\tilde{q}, \tilde{r})$ (satisfying \eqref{S-admissible}), there exists $C>0$, independent of $h$, such that
	\begin{equation}\label{Strichartz estimate 1}
	\| e^{it\Delta_h} u_0\|_{L_t^q(\mathbb{R};L_h^r)}\leq C \||\nabla_h|^{\frac{1}{q}}u_0\|_{L_h^2}
	\end{equation}
	and
	\begin{equation}\label{Strichartz estimate 2}
	\left\|\int_0^t e^{i(t-s)\Delta_h}F(s)ds\right\|_{L_t^q(\mathbb{R};L_h^r)}\leq C \||\nabla_h|^{\frac1q+\frac{1}{\tilde{q}}}F\|_{L_t^{\tilde{q}'}(\mathbb{R};L_h^{\tilde{r}'})}.
	\end{equation}
	Moreover, these inequalities are optimal in the sense that the range of $(q,r)$ cannot be extended and for fixed $(q,r)$ the required derivative loss is essential, as long as $h$ uniform estimates are concerned. For a precise statements, see Proposition~\ref{Prop:optimal}.
\end{thm}

\begin{rem}
	$(i)$ On $\mathbb{R}^d$, combining the Sobolev inequality and the Strichartz estimates in Keel-Tao \cite{KT98}, the following Sobolev-Strichartz estimates are available,
	\begin{equation}\label{Sobolev-Strichartz estimates}
	\| e^{it\Delta} u_0\|_{L_t^{q_*}(\mathbb{R};L_x^{r_*}(\mathbb{R}^d))}\lesssim \||\nabla|^s e^{it\Delta} u_0\|_{L_t^{q_*}(\mathbb{R};L_x^{\frac{dr_*}{d-sr_*}}(\mathbb{R}^d))}\lesssim \||\nabla|^su_0\|_{L_x^2(\mathbb{R}^d)},
	\end{equation}
	where $2\leq q_*\leq\infty$, $2\leq r_*<\infty$, $0\leq s<\frac{d}{2}$ and
	$$\frac{2}{q_*}+\frac{d}{r_*}=\frac{d}{2}-s.$$
	Here, $(q_*,r_*)$ lies on the painted trapezoid in Figure~\ref{S-pair}. The Strichartz estimate \eqref{Strichartz estimate 1} corresponds to the red line in Figure~\ref{S-pair} in the ``formal" limit $h\to 0$.\\
	$(ii)$ In the Strichartz estimates \eqref{Strichartz estimate 1}, the admissible conditions \eqref{S-admissible} must be satisfied due to the weaker dispersion \eqref{SK dispersion estimate} for each $h>0$. Thus, the presence of the derivative $|\nabla|^{1/q}$ cannot be avoided in the connection to its formal continuum limit \eqref{Sobolev-Strichartz estimates}. 
\end{rem}

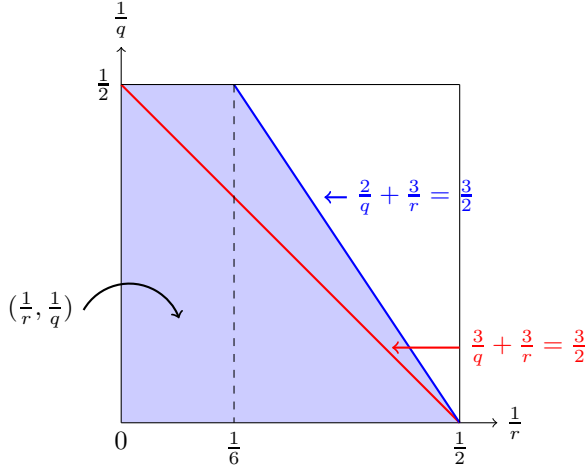
\begin{figure}\label{S-pair}
  \begin{tikzpicture}[domain=0:5] 
\fill[color=blue!20] (0,0) -- (4.5,0) -- (1.5,4.5) -- (0,4.5) ;
\draw[->] (0,0) -- (5,0) node[right] {$\frac1r$}; 
\draw[->] (0,0) -- (0,5) node[above] {$\frac1q$};
\draw[thick, color=red, domain=0:4.5]    plot (\x,-\x+4.5) ; 
\draw[thick, color=blue, domain=1.5:4.5]   plot (\x,{-1.5*\x+6.75}) ;
\draw[color=black, domain=4.5:0] plot (\x,4.5) node[left] {$\frac12$};
\draw[color=black, domain=4.5:0] plot (4.5,\x) node[below] {$\frac12$};
\draw[thick,red,<-] (3.6,1) -- (4.5,1) node[right] {$\frac3q+\frac3r=\frac32$} ;
\draw[thick,blue,<-] (2.7,3) -- (3,3) node[right] {$\frac2q+\frac3r=\frac32$} ;
\draw[dashed] (1.5,4.5) -- (1.5,0) node[below] {$\frac16$};
\node[below] at (0,0) {0};
\draw [thick,->] (-0.5,1.5) arc (150:20:20pt) node[at start,left] {$(\frac1r,\frac1q)$};
\end{tikzpicture}
\caption{Strichartz estimates $(q,r)$ pair for $d=3$. }
\end{figure}

We prove the Strichartz estimates \eqref{Strichartz estimate 1} (as well as \eqref{Strichartz estimate 2}) separating the \textit{bad} high frequency part from the \textit{good} low frequency part. Indeed, by the lattice Fourier transform, the solution $e^{it\Delta_h} u_0$ can be written as the following oscillatory integral
\begin{equation}\label{oscillatory integral}
e^{it\Delta_h}u_0(x)=\frac{1}{(2\pi)^d}\int_{\frac{2\pi}{h}[-\frac{1}{2},\frac{1}{2}]^d} e^{-i(\frac{4t}{h^2} \sum_{j=1}^d \sin^2(\frac{h\xi_i}{2})-x\cdot\xi)}\hat{u}_0(\xi)d\xi.
\end{equation}
Here, the phase function
$$\phi(\xi):=-\frac{4t}{h^2} \sum_{j=1}^d \sin^2\left(\frac{h\xi_j}{2}\right)+x\cdot\xi=-\frac{2t}{h^2} \sum_{j=1}^d (1-\cos(h\xi_j))+x\cdot\xi$$
has degenerate Hessian if and only if $\xi_j=\pm\frac{\pi}{2h}$ for some $j$ (in Figure~2, they correspond to the dashed line). For the high frequency part where the low frequencies, i.e., $\frac{2\pi}{h}[-\frac{1}{8},\frac{1}{8}]^d$, are smoothly truncated out, we reduce to the problem on $\mathbb{Z}^d$ from that on $\mathbb{Z}_h^d$ by a simple scaling argument, and make use of the previously-known result \eqref{SK Strichartz1} to get the desired bound. For the remaining low frequency part where the Hessian of the phase function is non-degenerate, we decompose the frequency domain dyadically by the Littlewood-Paley projections, and estimate each piece. Finally, summing up, we complete the proof.

\begin{figure}
	\begin{tikzpicture}
	\draw[dashed] (-1,-1) rectangle (1,1);
	\filldraw (0,0) circle (1pt)  node[below] {0};
	\draw (-2,-2) rectangle (2,2);
	\draw[<->] (0,-2.3) -- (2,-2.3) node[below, pos=0.5] {$\frac{\pi}{h} $};
	\draw[<->] (0,0) -- (1,0) node[below, pos=0.5] {$\frac{\pi}{2h} $};
	\draw [thick,->] (-2.3,1.1) arc (150:20:20pt) node[at start,left] {$\text{degenerate}$};
	\end{tikzpicture}
	\caption{Degenerate point in Fourier side for $d=2$} 
\end{figure}
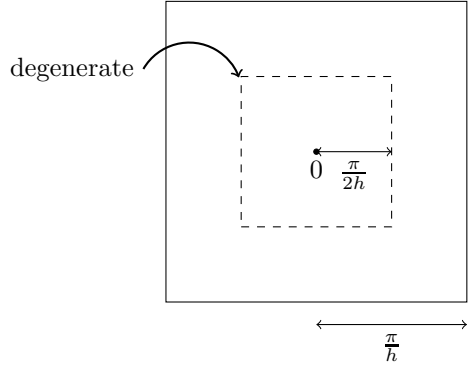

In order to sum the frequency pieces, we should employ the Littlewood-Paley inequality on a lattice $\mathbb{Z}_h^d$ that again holds uniformly in $h\in(0,1]$. However, to the best of authors' knowledge, not only the Littlewood-Paley inequality but the Calderon-Zygmund theory and the H\"ormander-Mikhlin theorem on a lattice are also not written in the literature although their proofs are pretty standard. Thus, a large portion of this paper is devoted to writing them down, which by itself would be of interest in an analysis point of view. The Calderon-Zygmund theory is also employed to prove norm equivalence (Proposition \ref{norm equivalence}) as boundedness of the Riesz transforms is proved on $\mathbb{R}^d$.

We now consider a discrete nonlinear Schr\"odinger equation (NLS)
\begin{equation}\label{eq: discrete NLS}
\left\{\begin{aligned}
i\partial_tu_h+\Delta_h u_h+\lambda |u_h|^{p-1}u_h&=0,\\
u_h(0)&=u_{h,0}, 
\end{aligned}\right.
\end{equation}
where $u_h=u_h(t,x):\mathbb{R}\times\mathbb{Z}_h^d\to\mathbb{C}$. It is not difficult to show that for fixed $h>0$ the equation \eqref{eq: discrete NLS} is globally well-posed in $L_h^2$, and that its solutions conserve the mass
	\begin{equation}\label{Conservation:Mass}
	M(u):= \|u\|_{L_h^2}^2=h^d\sum_{x\in\mathbb{Z}_h^d}|u(x)|^2  
	\end{equation}
	and the energy
	\begin{equation}\label{Conservation:Energy}
	E(u):=\frac12\|\sqrt{-\Delta_h} u\|_{L_h^2}^2+ \frac{\lambda}{p+1}\| u \|_{L_h^{p+1}}^{p+1}=h^d\sum_{x\in\mathbb{Z}_h^d}\frac{1}{2}(-\Delta_h u(x))\overline{u(x)}+\frac{\lambda}{p+1}|u(x)|^{p+1}
	\end{equation}
(see Proposition \ref{GWP}). It follows from the mass conservation law and the inequality $\|u\|_{L_h^\infty}\leq C_h\|u\|_{L_h^2}$ that solutions to \eqref{eq: discrete NLS} are bounded in $L_h^r$ for all $r\geq 2$. Nevertheless, their upper bounds may depend on $h>0$. 

As an application of Theorem \ref{thm: Strichartz for Schrodinger}, we prove that the higher $L_h^r$-norms of solutions are uniformly bounded in a time average sense. Precisely, we prove that if initial data are bounded uniformly in $h\in(0,1]$, then their solutions are bounded in the Strichartz norm
\begin{equation}\label{Strichartz norm}
\|u\|_{S_h^1(I)}:=\left\{\begin{aligned}
&\sup\left\{\|u\|_{L_t^q(I; W_h^{1-\frac{1}{q},r})}: (q,r)\textup{ satisfies }\eqref{S-admissible}\right\}&&\textup{if }d=1,2,\\
&\sup\left\{\|u\|_{L_t^q(I; W_h^{1-\frac{1}{q},r})}: (q,r)\textup{ satisfies }\eqref{S-admissible}\textup{ and }2\leq r\leq\infty^-\right\}&&\textup{if }d=3,
\end{aligned}\right.
\end{equation}
where $\infty^-$ denotes a preselected arbitrarily large number\footnote{In three dimensions, we additionally assume that $r\leq\infty^-$ just for a technical reason. Indeed, when $d=3$, the range of $\frac{1}{r}$ for admissible pairs, $(0,\frac{1}{2}]$, is not closed so the Strichartz norm may not be defined properly without restricting it to a compact interval $[\frac{1}{\infty^-},\frac{1}{2}]$.}.

\begin{thm}[Improved uniform bound]\label{thm: uniform bound}
Suppose that $0<h\leq1$ and $1\leq d\leq 3$. Let $u_h(t)\in C_t(\mathbb{R};L_h^2)$ be the global solution to NLS \eqref{eq: discrete NLS} with initial data $u_{h,0}\in L_h^2$.
\begin{enumerate}[(i)]
	\item (Improved uniform bound) If
$$\sup_{h\in(0,1]}\|u_{h,0}\|_{H_h^1}\leq R,$$
then there exists an interval $I\subset\mathbb{R}$ such that 
	\begin{equation}\label{improved uniform bound}
	\sup_{h\in(0,1]}\|u_h\|_{S^1(I)}<\infty.
	\end{equation}
	\item (Global-in-time uniform bound) Let $I_{max}$ be the maximal interval of uniform boundedness, that is, the largest interval such that \eqref{improved uniform bound} holds on any compact interval $I\subset I_{max}$. If $\lambda=1$ and $\max\{\frac{d-2}{d+2},0\}<\frac1p<1$ or if $\lambda=-1$ and $1<p<1+\frac{4}{d}$, then $I_{max}=\mathbb{R}$.
\end{enumerate}
\end{thm}

\begin{rem}
	In spite of the presence of the derivative on the right hand side in Strichartz estimates \eqref{Strichartz estimate 1} and \eqref{Strichartz estimate 2}, we can still recover the optimal local theory in the discrete setting.
\end{rem}

Finally, applying the aforementioned strategy to the discrete Klein-Gordon equation \eqref{eq: discrete KG}, we prove the following.

\begin{thm}[Strichartz estimates for the linear Klein-Gordon equation]\label{thm: Strichartz for KG}
For one dimensional discrete Schr\"odinger admissible pair $(q,r)$, there exist a constant $C$ independent of $h$ such that 
\begin{equation}
	\| e^{it\sqrt{1-\Delta_h}} u_0 \|_{L_t^q L_h^r (\R\times \Z_h)}
	\le C \| |\nabla_h|^\frac13\la \nabla_h \ra u_0 \|_{L_h^2(\Z_h)}.
\end{equation}
\end{thm}

\begin{rem}
	The case $d\geq 3$ is an interesting open question.
\end{rem}

\subsection{Organization of the paper}
The organization of this paper is as follows. In Section~2, we provide definition and basic properties about functions on a lattice point. 
In Section~3, we extend the Calderon-Zygmund theory on $\R^d$ to lattice $\Z_h^d$.
Then, in Section~4, H\"ormander-Mikhlin multiplier theorem and its applications, like Littlewood-Paley theorem and Sobolev norm equivalence, are established.
In Section~5, we prove the uniform Strichartz estimates with the help of the tools developed in previous sections, which is our main theorem.
In Section~6, we provide the global well-posedness for Schr\"odinger equations and obtain improved uniform bound for solutions.
In Section~7, we consider the Kelin-Gordon equation and show the uniform Strichartz estimates. 
In appendix we address the sharpness of the uniform Strichartz estimates for Schr\"odinger case.
\subsection{Acknowledgment}
This research of the first author was supported by Basic Science Research Program through the National Research Foundation of Korea(NRF) funded by the Ministry of Education (NRF-2017R1C1B1008215). The second author was supported in part by Samsung Science and Technology Foundation under Project Number SSTF-BA1702-02.

\section{Preliminaries}

In this section, we briefly introduce the preliminary $L^p$ theory, the Fourier transform and some elementary inequalities on a lattice (see also Section 17 in \cite{C-14}).

\subsection{$L_h^p$ spaces and basic inequalities}

Fix $h\in(0,1]$. For $1\leq p\leq\infty$, the $L_h^p$-space is the function space equipped with the norm $\|\cdot\|_{L_h^p}$ (see \eqref{L^p norm}). For $1\leq p<\infty$, the weak $L_h^p$-space, denoted by $L_h^{p,\infty}$, is defined as the collection of complex-valued functions such that
$$\|f\|_{L_h^{p,\infty}}:=\sup_{\lambda>0}\lambda\left|\left\{x\in\mathbb{Z}_h^d: |f(x)|\geq\lambda\right\}\right|^{1/p}<\infty,$$
where for a set $A\in\mathbb{Z}_h^d$, $|A|$ denotes the standard normalized counting measure on a lattice, i.e., $|A|=h^d\sum_{x\in A}1$. On a lattice $\mathbb{Z}_h^d$, we define the inner product by
$$ \la f, g\ra=\la f, g\ra_h:= h^d \sum_{x\in\Z_h^d} f(x)\overline{g(x)}$$
and the convolution by 
$$(f*g)(x)=(f*_hg)(x):= h^d \sum_{y\in\Z_h^d} f(x-y)g(y).$$
In the following propositions, we collect some basic inequalities and the interpolation theorems, whose proofs are omitted here because they follow from the standard arguments.

\begin{prop}
\begin{enumerate}[(i)]
	\item (H\"older inequality) If $\frac{1}{p}=\frac{1}{p_1}+\frac{1}{p_2}$, then
	\begin{align*}
	\|fg\|_{L_h^p} \le \|f\|_{L_h^{p_1}}\|g\|_{L_h^{p_2}}.
	\end{align*} 
	\item (Duality) If $1\leq p\leq\infty$, then
	$$\|f\|_{L_h^p} =\sup_{\|g\|_{L_h^{p'}}\le 1} h^d\sum_{x\in \mathbb{Z}_h^d}f(x)\overline{g(x)}.$$
	Moreover, if $1\leq p<\infty$, then $(L_h^p)^*=L_h^{p'}$, where $\frac{1}{p}+\frac{1}{p'}=1$.
	\item (Young's convolution inequality) If $1+\frac{1}{r}=\frac{1}{p}+\frac{1}{q}$, then
	\begin{align*}
	\|f*g\|_{L_h^r}\leq \|f\|_{L_h^p}\|g\|_{L_h^q}.
	\end{align*}
\end{enumerate}
\end{prop}

\begin{prop}[Real interpolation]
Suppose that $\frac{1}{p_\theta}=\frac{\theta}{p_0}+\frac{1-\theta}{p_1}$, $\frac{1}{q_\theta}=\frac{\theta}{q_0}+\frac{1-\theta}{q_1}$ for some $\theta\in(0,1)$. Then, if a sublinear operator $T$, acting on $L_h^{p_0}+L_h^{p_1}$, satisfies 
$$\|Tf\|_{L_h^{q_j, \infty}}\leq C_j\|f\|_{L_h^{p_j}}\quad\textup{for }j=0,1,$$
then 
$$\|Tf\|_{L_h^{q_\theta}}\leq C_0^\theta C_1^{1-\theta}\|f\|_{L_h^{p_\theta}}.$$
\end{prop}

\subsection{Fourier transform}

\begin{figure}[b]
	\begin{tikzpicture}
	\foreach \x in {0,...,5}  \filldraw (\x,0) circle (0.8pt);
	\foreach \x in {0,...,5}  \filldraw (\x,1) circle (0.8pt);
	\foreach \x in {0,...,5}  \filldraw (\x,2) circle (0.8pt);
	\foreach \x in {0,...,5}  \filldraw (\x,3) circle (0.8pt);
	\foreach \x in {0,...,5}  \filldraw (\x,4) circle (0.8pt);
	\draw[thick,dashed,<->] (0,0) -- (0,1) node[left, pos=0.5] {h};
	\draw[thick,dashed,<->] (0,0) -- (1,0) node[above, pos=0.5] {h};
    \node at (3,2.7)  {$x$} ;
	\node[thick] at (2.5,-0.8)  {$\Z_h^2$} ;
	\draw[thick,<->] (6.2,2) -- (8.3,2) node[align=center,above, pos=0.5] {Fourier \\ transform};
	\draw (9.5,0) rectangle (13.5,4);
	\draw[thick,dashed,<->] (9.5,-0.2) -- (13.5,-0.2) node[below,pos=0.8] {$\frac{2\pi}{h}$};
	\draw[thick,dashed,<->] (13.7,4) -- (13.7,0) node[right,pos=0.5] {$\frac{2\pi}{h}$};
	\filldraw (11.5,2)  circle(1pt) node[below] {0};
	\node[thick] at (11.5,-0.8)  {$\mathbb{T}_h^2$} ;
	\filldraw (12,3.1) circle (0.5pt) node[below] {$\xi$};
	\end{tikzpicture}
	\caption{Domain in lattice and Fourier side for $d=2$} 	
\end{figure}
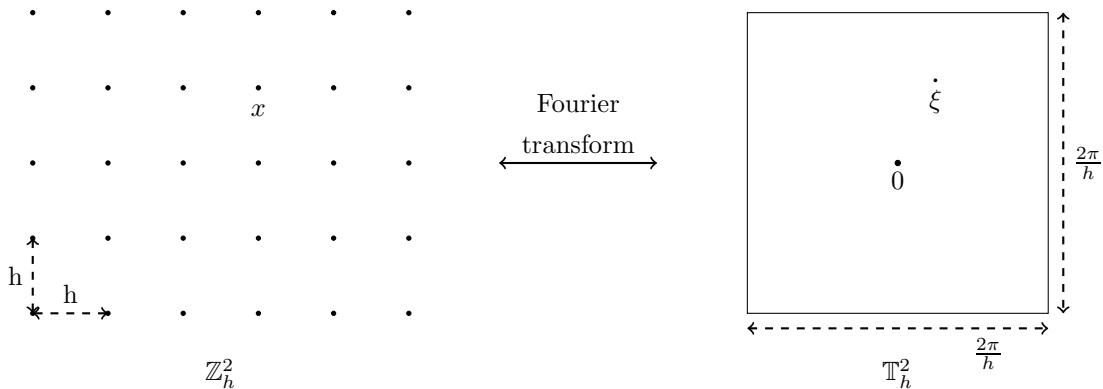

For a rapidly decreasing function $f: \mathbb{Z}_h^d\to \mathbb{C}$, that is, $|x|^k |f(x)|\leq C_k$ for all $k=0,1,2, ...$, we define its Fourier transform by
$$\hat{f}(\xi):=(\mathcal{F}_h f)(\xi)=h^d\sum_{x\in \mathbb{Z}_h^d} f(x)e^{-ix\cdot\xi},\quad\xi\in\mathbb{T}_h^d,$$
where $\tph=\frac{2\pi}{h}[-\frac{1}{2},\frac{1}{2}]^d$. The inverse Fourier transform of a smooth function $f:\mathbb{T}_h^d\to\mathbb{C}$ is defined by
$$\check{f}(x):=(\mathcal{F}_h^{-1} f)(x)=\frac{1}{(2\pi)^d}\int_{\mathbb{T}^d_h} f(\xi) e^{ix\cdot\xi}d\xi,\quad x\in \mathbb{Z}_h^d.$$
Indeed, since we here consider functions on the lattice $\mathbb{Z}_h^d$, the Fourier transform is defined in the opposite way to what is done for periodic functions as Fourier series (see Figure~3). We also note that the Fourier and the inverse Fourier transforms formally converge to those on the whole space $\mathbb{R}^d$ in the continuum limit $h\to 0$:
$$(\mathcal{F}_hf)(\xi)\to \int_{\mathbb{R}^d}f(x)e^{-ix\cdot\xi}dx,\quad\xi\in\mathbb{R}^d,$$
and
$$(\mathcal{F}_h^{-1}f)(x)\to\frac{1}{(2\pi)^d}\int_{\mathbb{R}^d} f(\xi) e^{ix\cdot\xi}d\xi,\quad x\in \mathbb{R}^d.$$

The Fourier transform (respectively, its inversion) can be extended to a larger class of functions, that is, the dual space of rapidly decreasing functions (respectively, that of smooth functions) via the duality relation 
$$\int_{\mathbb{T}_h^d} \hat{f}(\xi)\overline{g(\xi)} d\xi=h^d\sum_{x\in \mathbb{Z}_h^d} f(x)\overline{\check{g}(x)}.$$
Moreover, we have:
\begin{align*}
\widehat{f*g}(\xi)&=\hat{f}(\xi)\hat{g}(\xi)&&\textup{(Fourier transform of convolution)}\\
\|\hat{f}\|_{L^p(\mathbb{T}_h^d)}&\lesssim \|f\|_{L_h^{p'}},\quad\forall p\geq 2&&\textup{(Hausdorff-Young)}\\
\frac{1}{(2\pi)^d}\int_{\mathbb{T}_h^d}\hat{f}(\xi)\overline{\hat{g}(\xi)}d\xi&=h^d\sum_{x\in \mathbb{Z}_h^d} f(x)\overline{g(x)}&&\textup{(Parseval)}
\end{align*}

\subsection{Littlewood-Paley projections}
Let $\phi:\mathbb{R}^d\to[0,1]$ be an axisymmetric smooth bump function such that $\phi(\xi)\equiv 1$ on the square $[-1,1]^d$ but $\phi(\xi)\equiv 0$ on $\mathbb{R}^d\setminus [-2,2]^d$, and let $\varphi:=\phi-\phi(2\cdot)$. For a dyadic number $N\in2^{\mathbb{Z}}$, we denote 
$$\psi_N(\xi)=\psi_{N;h}(\xi):=\varphi\left(\frac{h\xi}{2\pi N}\right).$$
Then, we have
$$\supp\psi_{N}\subset[-\tfrac{4\pi N}{h}, \tfrac{4\pi N}{h}]^d\setminus [-\tfrac{\pi N}{h}, \tfrac{\pi N}{h}]^d$$
and
$$\sum_{N\leq 1}\psi_{N}\equiv 1\quad\textup{on }\mathbb{T}_h^d,$$
where with an abuse of notation, $\psi_{1}$ denotes the function $\psi_{1}$ restricted to the frequency domain $\mathbb{T}_h^d$.

We now define the Littlewood-Paley projection operator $P_N=P_{N;h}$ as a Fourier multiplier such that
$$\widehat{P_{N}f}(\xi)=\psi_{N}(\xi)\hat{f}(\xi)\quad\textup{on }\mathbb{T}_h^d,$$
where $\hat{\cdot}$ stands for the Fourier transform on the lattice $\mathbb{Z}_h^d$. Note that unlike the usual definition of Littlewood-Paley projections on $\mathbb{R}^d$, $\psi_N$ is not supported on an annulus on the Fourier side, because the entire Frequency space is a cube $[-\frac{\pi}{h},\frac{\pi}{h}]^d$.

As an analogue of the classical theory on the whole space $\mathbb{R}^d$, the Littlewood-Paley projections satisfy the following boundedness property.
\begin{lem}[Bernstein's inequality]\label{lem:Bernstein}
Let $h\in(0,1]$. If $1\le p\leq q\le\infty$, then for any dyadic number $N\in 2^\Z$ with $N\leq 1$, we have
\begin{align}\label{ineq:Bernstein}
\| P_N f\|_{L_h^q}  \ls \left(\frac{N}{h}\right)^{d(\frac1p-\frac1q)}\|f\|_{L_h^p},
\end{align}
where the implicit constant is independent of $N$ and $h$.
\end{lem}

\begin{proof}
It follows from Young's inequality that
$$\| P_N f\|_{L_h^q} =\| (\psi_{N})^\vee *f \|_{L_h^q}\lesssim \| (\psi_{N})^\vee\|_{L_h^r} \|f\|_{L_h^p},$$
where $1+\frac1q=\frac1r+\frac1p$. Thus, it suffices to show that $\| (\psi_{N})^\vee\|_{L_h^r}\lesssim \left(\frac{N}{h}\right)^{d(\frac1p-\frac1q)}=\left(\frac{N}{h}\right)^{\frac{d}{r'}}$. Indeed, by change of variables, we have
\begin{align*}
 \| (\psi_{N})^\vee\|_{L_h^r}& =\left\| \int_{\mathbb{T}_h^d} e^{ix\cdot \xi} \varphi\left(\frac{h\xi}{2\pi N}\right) d\xi \right\|_{L_h^r} =h^{\frac{d}{r}}\left(\frac{2\pi}{h}\right)^d\left\{\sum_{x\in\mathbb{Z}^d}\left|\int_{\mathbb{T}^d} e^{2\pi ix\cdot\xi} \varphi\left(\frac{\xi}{N}\right)d\xi\right|^r\right\}^{1/r}.
\end{align*}
Then, a simple integration by parts using $e^{2\pi ix\cdot\xi}=\frac{1}{2\pi ix_j}\partial_{\xi_j}e^{2\pi ix\cdot\xi}$ deduces that 
$$\left|\int_{\mathbb{T}^d} e^{2\pi ix\cdot\xi} \varphi\left(\frac{\xi}{N}\right)d\xi\right|\lesssim\frac{N^d}{(1+|Nx|)^{d+1}}$$
Therefore, we conclude that 
\begin{align*}
 \| (\psi_{N;h})^\vee\|_{L_h^r}&\lesssim h^{-\frac{d}{r'}}\left\{\int_{\R^d} \frac{N^{dr}}{(1+|Nx|)^{(d+1)r}}dx\right\}^{1/r}\sim \left(\frac{N}{h}\right)^{\frac{d}{r'}}.
\end{align*}
\end{proof}

\subsection{Sobolev embedding and Gagliardo-Nirenberg inequality}
Using Bernstein's inequality, we deduce the Gagliardo-Nirenberg inequality.

\begin{prop}[Gagliardo-Nirenberg inequality]
	Let $h\in(0,1]$. If $1\leq p\leq q\leq\infty$, $0<\theta<1$ and $\frac{1}{q}=\frac{1}{p}-\frac{\theta s}{d}$, then
	\begin{equation}\label{Inequality: Gagliardo-Nireberg}
	\|f\|_{L_h^q}\lesssim \|f\|_{L_h^p}^{1-\theta}\|f\|_{\dot{W}_h^{s,p}}^\theta,
	\end{equation}
	where the implicit constant is independent of $h$.
\end{prop}

\begin{proof}
Replacing $f$ by $\frac{1}{\|f\|_{L_h^p}}f$, we may assume that $\|f\|_{L_h^p}=1$. Suppose that $\|f\|_{\dot{W}_h^{s,p}}\leq h^{-s}$. Let $R=h\|f\|_{\dot{W}_h^{s,p}}^{1/s}$. Then, we write 
	\begin{align*}
	\|f\|_{L_h^q}&\leq \sum_{N\leq R}\|P_Nf\|_{L_h^q}+\sum_{R<N\leq 1}\|P_Nf\|_{L_h^q}\\
	&=\sum_{N\leq R}\|P_Nf\|_{L_h^q}+\sum_{R<N\leq 1}\left(\frac{N}{h}\right)^{-s}\|\tilde{P}_N(|\nabla|^sf)\|_{L_h^q}
	\end{align*}
	where $\tilde{P}_N$ is the Fourier multiplier of the symbol $\widetilde{\psi_N}=\psi_{N-1}+\psi_N+\psi_{N+1}$ which is identity on the support of $\psi_N$ and supported near $\frac Nh$.
	Hence, by Bernstein's inequality (Lemma \ref{lem:Bernstein}), we prove that 
	\begin{align*}
	\|f\|_{L_h^q}&\lesssim\sum_{N\leq R}\left(\frac{N}{h}\right)^{s\theta}\|f\|_{L_h^p}+ \sum_{R\leq N\leq 1}\left(\frac{N}{h}\right)^{-s(1-\theta)} \|f\|_{\dot{W}_h^{s,p}}\\
	&\sim \left(\frac{R}{h}\right)^{s\theta}+\left(\frac{R}{h}\right)^{-s(1-\theta)} \|f\|_{\dot{W}_h^{s,p}}\sim \|f\|_{\dot{W}_h^{s,p}}^{\theta}.
	\end{align*}
	Similarly, if $\|f\|_{\dot{W}_h^{s,p}}\geq h^{-s}$, then 
	$$\|f\|_{L_h^q}\leq \sum_{N\leq 1}\|P_Nf\|_{L_h^q}\lesssim\sum_{N\leq 1} \left(\frac{N}{h}\right)^{s\theta}\|f\|_{L_h^p}\sim h^{-s\theta}=\|f\|_{\dot{W}_h^{s,p}}^\theta.$$
\end{proof}

As a consequence, we derive the Sobolev inequality except the sharp exponent.
\begin{prop}[Non-endpoint Sobolev inequality]\label{Ineq: Sobolev infinity}
	Let $h\in(0,1]$. If $1\leq p<q\leq\infty$ and $\frac{1}{q}>\frac{1}{p}-\frac{s}{d}$, then
	$$\|f\|_{L_h^q}\lesssim \|f\|_{W_h^{s,p}},$$
	where the implicit constant is independent of $h$.
\end{prop}

\begin{proof}
	By the assumption, there exists $\theta\in(0,1)$ such that $\frac{1}{q}=\frac{1}{p}-\frac{\theta s}{d}$. Hence, it follows from the Gagliardo-Nirenberg inequality \eqref{Inequality: Gagliardo-Nireberg} that $$\|f\|_{L_h^q}\lesssim \|f\|_{L_h^p}^{1-\theta}\|f\|_{\dot{W}_h^{s,p}}^\theta\lesssim \|f\|_{W_h^{s,p}}, $$
	where the last inequality we used \eqref{ineq:bound by inhomogeneous}.
\end{proof}

\section{Calderon-Zygmund theory on a lattice}

We consider convolution operators on a lattice of the form
$$T_{K_h}f(x):=h^d\sum_{y\in\mathbb{Z}_h^d} K_h(x-y)f(y).$$
Such operators are of course bounded on $L_h^p$ by Young's inequality, because there is no singular kernel on a lattice. However, getting a uniform-in-$h$ bound is not so obvious. In this section, we show boundedness of convolution operators satisfying the hypotheses similar to those for Calderon-Zygmund operators in the formal limit $h\to 0$, extending the Calderon-Zygmund theory on the Euclidean space $\mathbb{R}^d$ 

\begin{thm}[Calder\'on-Zygmund]\label{thm:CZ}
Suppose that for all $h\in(0,1]$,  
\begin{equation}\label{kernel1}
|\hat{K}_h(\xi)| \le A\quad\textup{for all }\xi\in\mathbb{T}_h^d,
\end{equation}
\begin{equation}\label{kernel2}
h^d\sum_{y\in\mathbb{Z}_h^d, |y|\geq 2|x|} |K_h(y-x)-K_h(y)| \le B.
\end{equation}
Then, for $1<p<\infty$, there exists $C_p>0,$ independent of $h\in(0,1]$, such that
\begin{equation}\label{eq: CZ thm}
\|T_{K_h}f\|_{L_h^p}\le C_p\|f\|_{L_h^p}.
\end{equation}
\end{thm}

We prove Theorem \ref{thm:CZ} following the standard argument for instance in \cite{D-01}, involving the dyadic maximal function. We fix $h\in(0,1]$. For each dyadic number $N\in 2^{\mathbb{Z}}$ with $N\ge1$, let $\mathcal{Q}_N=\mathcal{Q}_{N;h}$ be the family of cubes, open on the right, whose vertices are adjacent points of the lattice $hN\Z^d$. Given $f\in L_h^1$, averaging over each cube in $\mathcal{Q}_N$, we introduce the average function  
$$E_Nf(x)=E_{N;h} f(x):=\sum_{Q\in \mathcal{Q}_{N} } \left( \frac{1}{N^d}\sum_{y\in Q}f(y) \right)\mathbf{1}_{Q}(x).$$
Note that $\mathcal{Q}_{1}\approx\mathbb{Z}_h^d$ and $E_{1}f(x)=f(x)$. Next, we define the dyadic maximal function by 
\begin{align*}
M f(x)=M_hf(x):=\sup_{N\ge 1} |E_N f(x)|.
\end{align*}
Using this maximal function, we decompose the domain of a function. 

\begin{thm}[Calder\'on-Zygmund decomposition]
Given non-negative $f \in L_h^1$ and $\lambda>0$, there exists a collection $\{ Q_k\}_k$ of disjoint dyadic cubes such that  	
\begin{enumerate}\label{C-Z decomposition}
\item $f(x)\le \lambda$ for every $\displaystyle x \notin \bigsqcup_k Q_k$;
\item $\displaystyle\big| \bigsqcup_k Q_k \big| \le \frac1\lambda \| f \|_{L_h^1}$;
\item $\displaystyle\lambda <\frac{1}{|Q_k|}\cdot h^d\sum_{x\in Q_{k}}f(x) \le 2^d\lambda.$
\end{enumerate}
\end{thm}

\begin{proof}
In order to construct the desired dyadic cubes, we claim that 
\begin{equation}\label{CZ decomposition proof}
\{ x \in \Z_h^d : M f(x) >\lambda \} = \bigsqcup_N \Omega_N,
\end{equation}
where 
\begin{align*}
\Omega_N=\{ x\in \Z_h^d : E_{N}f(x) >\lambda \ \text{and} \ E_{M}f (x) \le \lambda  \ \text{if} \ M>N \}.
\end{align*}
Note that for $x\in\Omega_N$, $N$ is the smallest dyadic numbers such that $E_Nf(x)>\lambda$. Hence, it is obvious that $\Omega_N$'s are disjoint each other. The inclusion $\supset$ in \eqref{CZ decomposition proof} is trivial by the definition of the maximal function. For the opposite inclusion, we observe that $E_{N}f \rightarrow 0$ as $N \rightarrow\infty$, because $f \in L_h^1$. Thus, for any $x\in\mathbb{Z}_h^d$ with $Mf(x)>\lambda$, there exists $N$ such that $x\in\Omega_N$.

By the definition, $E_Nf(x)$ has the same value on a dyadic cube containing $x$. Thus, each $\Omega_N$ can be decomposed into disjoint cubes contained in $\mathcal Q_{N,h}$. Thus, collecting all disjoint dyadic cubes, we may write 
$$\{ x \in \Z_h^d : M f(x) >\lambda \} = \bigsqcup_k Q_k.$$

For $(1)$, we observe that if $x\notin \bigsqcup_k Q_k$, then $E_Nf(x)\leq\lambda $. In particular, $E_{1}f(x)=f(x)\leq\lambda$. For $(2)$, we use the decomposition \eqref{CZ decomposition proof} to get 
$$|\{ x\in \Z_h^d : M f(x)>\lambda \}|=\sum_{N} |\Omega_N| =\sum_{N}h^d\sum_{x\in \Omega_N}1
\le \sum_{N} \frac{h^d}{\lambda}\sum_{x\in\Omega_N} E_{N} f (x).$$
If $\Omega_N=\bigsqcup_{k_N} Q_{k_N}$, then $E_Nf(x)$ has the same value on $Q_{k_N}$, which is $\frac{1}{N^d}\sum_{y\in Q_{k_N}}f(y)$. Thus, 
\begin{align*}
|\{ x\in \Z_h^d : M f(x)>\lambda \}|&\le \frac{h^d}{\lambda}\sum_{N} \sum_{k_N}\sum_{x\in Q_{k_N}} E_{N} f (x)= \frac{h^d}{\lambda}\sum_{N} \sum_{k_N}\sum_{x\in Q_{k_N}} \frac{1}{N^d}\sum_{y\in Q_{k_N}}f(y)\\
&= \frac{h^d}{\lambda}\sum_{N} \sum_{k_N}\sum_{y\in Q_{k_N}}f(y)= \frac{h^d}{\lambda}\sum_{y\in \bigsqcup \Omega_N}f(y)\leq\frac{1}{\lambda}\|f\|_{L_h^1}.
\end{align*}
It remains to show $(3)$. By the definition of the sets $\Omega_N$, the average of $f$ over $Q_k$ is greater than $\lambda$. Let $2Q_k$ be the dyadic cube containing $Q_k$ whose sides are twice as long. Then, the average of $f$ over $2Q_k$ is at most $\lambda$. Therefore, we prove that
\begin{align*}
\frac{h^d}{|Q_k|}\sum_{x\in Q_{k}}f(x)
\le \frac{h^d}{|Q_k|}\sum_{x\in 2Q_k}f(x)=2^d\cdot\frac{h^d}{|2Q_k|}\sum_{x\in 2Q_k}f(x)
\le 2^{d}\lambda.
\end{align*}
\end{proof}

Now we are ready to show Theorem \ref{thm:CZ}.

\begin{proof}[Proof of Theorem \ref{thm:CZ}]
By the Plancherel theorem with the bound \eqref{kernel1}, $T_{K_h}$ is bounded on $L_h^2$. Thus, it suffices to show that for arbitrary non-negative $f\in L_h^1$ and $\lambda>0$, 
\begin{equation}\label{CZ proof'}
\left|\left\{x\in\mathbb{Z}_h^d: T_{K_h}f(x)\geq\lambda\right\}\right|\leq\frac{C}{\lambda}\|f\|_{L_h^1}.
\end{equation}
Indeed, \eqref{CZ proof'} implies that $\|T_{K_h} f\|_{L_h^{1,\infty}}\leq C\|f\|_{L_h^1}$. Consequently by interpolation, \eqref{eq: CZ thm} holds for $1<p\leq 2$, and then for $2<p<\infty$ by duality. 

To show \eqref{CZ proof'}, applying the Calder\'on-Zygmund decomposition (Theorem~\ref{C-Z decomposition}) to given $f\in L_h^1$ and $\lambda>0$, we obtain the collection of disjoint dyadic cubes $\{Q_k\}_k$ with the desired properties, and then we decompose $f$ into the good function $g$ and the bad functions $b_k$'s:
$$f=g+b=g+\sum_k b_k$$
such that 
$$g(x)=\left\{\begin{aligned}
&f(x)&& \text{if} \  x\notin \bigsqcup_k Q_k,\\
&\frac{1}{|Q_k|} h^d\sum_{y\in Q_k} f (y) &&\text{if} \ x\in Q_k,
\end{aligned}\right.$$
and
$$b_k(x)=\left( f(x) - \frac{1}{|Q_k|} h^d\sum_{ y\in Q_k} f(y) \right)\mathbf{1}_{Q_k}(x).$$
For the good function, we observe from Theorem \ref{C-Z decomposition} that $g(x)\leq 2^d\lambda$. Hence, it follows from $L_h^2$ boundedness that 
\begin{align*}
\left|\left\{ x\in \Z_h^d: |T_{K_h}g(x)| >\frac{\lambda}{2} \right\}\right|&\leq\left(\frac{2}{\lambda}\right)^2 \|T_{K_h}g\|_{L_h^2}^2\\
&\leq \left(\frac{2}{\lambda}\right)^2C_2^2\|g\|_{L_h^2}^2=\left(\frac{2}{\lambda}\right)^2C_2^2h^d\sum_{x\in\mathbb{Z}_h^d}|g(x)|^2\\
&\le\left(\frac{2}{\lambda}\right)^2 C_2^2\cdot2^d\lambda\cdot h^d\sum_{x\in\mathbb{Z}_h^d}g(x)= \frac{2^{d+2}C_2^2}{\lambda}\|g\|_{L_h^1}.
\end{align*}
For the bad function, by a trivial estimate, we have
$$\left|\left\{ x\in\Z_h^d: |T_{K_h}b(x)|>\frac{\lambda}{2} \right\}\right|\le \big|\bigsqcup 2Q_k\big|+\left|\left\{ x\notin \bigsqcup 2Q_k: |T_{K_h}b(x)|>\frac{\lambda}{2} \right\}\right|,$$
where $2Q_k$ is the cube with the same center as $Q_k$ and twice the length. For the first term, by Theorem \ref{C-Z decomposition},
$$\big|\bigsqcup 2Q_k\big|=2^d\big|\bigsqcup Q_k\big|\leq\frac{2^d}{\lambda}\|f\|_{L_h^1}.$$
On the other hand, for the second term, we write
\begin{equation}\label{bad set estimate 1}
\begin{aligned}
\frac{\lambda}{2}\left|\left\{ x\notin \bigsqcup 2Q_k: |T_{K_h}b(x)|>\frac{\lambda}{2} \right\}\right|&\leq h^d\sum_{x\notin \bigsqcup 2Q_k}|T_{K_h}b(x)|\\
&\leq h^d\sum_k\sum_{x\notin \bigsqcup 2Q_k}|T_{K_h}b_k(x)|.
\end{aligned}
\end{equation}
We now recall that each $b_k$ is supported on $Q_k$ and that its average is zero, i.e., $h^d\sum b_k=0$. Moreover, by the triangle inequality, 
$$\|b_k\|_{L_h^1}\leq \|f\mathbf{1}_{Q_k}\|_{L_h^1}+\|f\mathbf{1}_{Q_k}\|_{L_h^1}\frac{1}{|Q_k|}\|\mathbf{1}_{Q_k}\|_{L_h^1}=2\|f\mathbf{1}_{Q_k}\|_{L_h^1}=2h^d\sum_{x\in Q_k}f(x).$$
Hence, we have
\begin{equation}\label{bad set estimate 2}
\begin{aligned}
h^d\sum_{x\notin \bigsqcup 2Q_k}|T_{K_h}b_k(x)|&=h^d\sum_{x\notin \bigsqcup 2Q_k} h^d\Big|\sum_{y\in Q_k} K(x-y)b_k(y)\Big|\\
&=h^d\sum_{x\notin \bigsqcup 2Q_k} h^d\Big|\sum_{y\in Q_k} \left(K(x-y)-K(x-y_k)\right)b_k(y)\Big|\\
&\leq h^{2d}\sum_{x\notin \bigsqcup 2Q_k} \sum_{y\in Q_k} |K(x-y)-K(x-y_k)||b_k(y)|\\
&=h^{2d} \sum_{y\in Q_k}\sum_{x\notin \bigsqcup 2Q_k} |K(x-y)-K(x-y_k)||b_k(y)|\\
&\leq B h^{d} \sum_{y\in Q_k}|b_k(y)|=B\|b_k\|_{L_h^1},
\end{aligned}
\end{equation}
where $y_k$ is the center of $Q_k$. Here, the property $h^d\sum_{k} b_k=0$ is used in the second identity, and the assumption \eqref{kernel2} is used in the last inequality. Therefore, going back to \eqref{bad set estimate 1} and summing \eqref{bad set estimate 2} in $k$, we prove that 
\begin{align*}
\left|\left\{ x\notin \bigsqcup 2Q_k: |T_{K_h}b(x)|>\frac{\lambda}{2} \right\}\right|\leq \frac{2B}{\lambda}\sum_k\|b_k\|_{L_h^1}\leq\frac{4B}{\lambda}\cdot h^d\sum_{x\in\cup_k Q_k} f(x)\leq \frac{4B}{\lambda}\|f\|_{L_h^1}.
\end{align*}
Finally, collecting all, we conclude that
\begin{align*}
\left|\left\{x\in\mathbb{Z}_h^d: T_{K_h}f(x)\geq\lambda\right\}\right|&\leq \left|\left\{ x\in \Z_h^d: |T_{K_h}g(x)| >\frac{\lambda}{2} \right\}\right|+\left|\left\{ x\in \Z_h^d: |T_{K_h}b(x)| >\frac{\lambda}{2} \right\}\right|\\
&\lesssim\frac{1}{\lambda}\|f\|_{L_h^1}.
\end{align*}
\end{proof}


\section{H\"ormander-Mikhlin theorem and its applications}

In this section, we present the H\"ormander-Mikhlin multiplier theorem on a lattice and its applications.

\subsection{H\"ormander-Mikhlin theorem}
Given a symbol function $m=m_h$ on $\tph$, we consider the Fourier multiplier operator $T_m$ defined by
$$\widehat{T_mf}(\xi)=m(\xi)\hat{f}(\xi)\quad\textup{on }\tph.$$
We show that this multiplier operator is uniformly bounded if the symbol satisfies the assumption completely analogous to that in the multiplier theorem on $\mathbb{R}^d$.

\begin{thm}[Hormander-Mikhlin]\label{thm: Hormander Mikhlin}
Let $h\in(0,1]$. Suppose that $m:\mathbb{Z}_h^d\to\mathbb{C}$ satisfies 
\begin{align}\label{Hormander condition}
| \nabla^\alpha m (\xi)| \le c_\alpha|\xi|^{-|\alpha|}\quad \textup{on }\mathbb{T}_h^d\setminus\{0\}
\end{align}for all multi-index $|\alpha|\le d+2$. Then, for $1<p<\infty$, there exists $C_p>0$, independent of $h$, such that 
$$\| T_m f\|_{L_h^p} \le C_p \|f\|_{L_h^p}.$$
\end{thm} 

\begin{proof}
Since $|m(\xi)|$ is bounded, it suffices to show that the integral kernel $\check{m}$ satisfies \eqref{kernel2}. We can naturally extend the kernel $\check{m}$ on $\Z_h^d$ to a function on $\R^d$.
By the Littlewood-Paley projections, we decompose
$$\nabla\check{m}(x)=\frac{1}{(2\pi)^d}\int_{\mathbb{T}_h^d} i\xi m(\xi) e^{ix\cdot\xi}d\xi=\sum_{N\leq 1}\frac{1}{(2\pi)^d}\int_{\mathbb{T}_h^d} i\xi m(\xi)\psi_N(\xi) e^{ix\cdot\xi}d\xi.$$
It is obvious that 
$$\left|\int_{\mathbb{T}_h^d} i\xi m(\xi)\psi_N(\xi) e^{ix\cdot\xi}d\xi\right|\lesssim\left(\frac{N}{h}\right)^{d+1}.$$
On the other hand, by integration by parts $(d+2)$ times with $e^{ix\cdot\xi}=\frac{1}{ix_j}\frac{\partial}{\partial\xi_j}e^{ix\cdot\xi}$, one can show that
$$\left|\int_{\mathbb{T}_h^d} i\xi m(\xi)\psi_N(\xi) e^{ix\cdot\xi}d\xi\right|\lesssim \frac{h}{N|x|^{d+2}}.$$
Summing in $N$, we get the bound,
$$|\nabla\check{m}(x)|\leq\min_{N\leq\frac{h}{|x|}}\left(\frac{N}{h}\right)^{d+1}+\min_{N>\frac{h}{|x|}}\frac{h}{N|x|^{d+2}} \sim \frac{1}{|x|^{d+1}}.$$
Using this, we finally check 
\begin{align*}
h^d\sum_{y\in\mathbb{Z}_h^d, |y|\geq 2|x|} |\check{m}(y-x)-\check{m}(y)| 
&\le h^d\sum_{y\in\mathbb{Z}_h^d, |y|\geq 2|x|} |\nabla \check{m} (y-xt) | |x| , \text{ for some } t\in[0,1] \\ 
&\ls h^d|x|\sum_{y\in\mathbb{Z}_h^d, |y|\geq 2|x|} |y|^{-d-1} \\
&\ls |x| \int_{|x|}^\infty |y|^{-d-1} dy = B,
\end{align*}
where the third one follows considering the Riemann summation and $B>0$ is independent of $h>0$.
\end{proof}

\subsection{Littlewood-Paley decomposition}
As a first application of the H\"ormander-Mikhlin theorem we show the Littlewood-Paley theorem is also valid for functions on lattice with uniform bound in $h$. 
\begin{thm}\label{thm:LP}
	Let $1<p<\infty$. For $f \in L_h^p$, there exist positive constants $c_p, C_p>0$, independent of $h\in(0,1]$, such that
	\begin{align}\label{ineq:LP}
	c_p\|f\|_{L_h^p}\le
	\Big\| \Big( \sum_N |P_N f|^2 \Big)^\frac12 \Big\|_{L_h^p} \le C_p \| f\|_{L_h^p}.
	\end{align}
\end{thm}

Our proof is based on the following randomization technique.
\begin{lem}[Khinchine's inequality for scalars \cite{Khinchiine1981best}]
Let $z_1,\cdots, z_N$ be complex numbers, and let $\epsilon_1,\cdots,\epsilon_N\in \{ -1, 1\}$ be independent random signs, drawn from $\{ -1, 1\}$ with the uniform distributions. Then for any $0<p<\infty$ 
\begin{align}\label{Khinchine's inequality}
\big( \mathbb E|\sum_{j=1}^N \epsilon_{j}z_{j}|^{p} \big)^{\frac1p}
\sim_{p} \big( \sum_{j=1}^N |z_j|^2 \big)^\frac12.
\end{align}
\end{lem}

\begin{lem}[Khinchine's inequality for functions on $\Z_h$]\label{Lem:Khinchine's inequality for functions} 
Let $f_1,\cdots,f_N\in L_h^p$ for some $1<p<\infty$, and let $\epsilon_1,\cdots,\epsilon_N\in \{ -1,1 \}$ be independent of random signs, drawn from 
$\{ -1, 1\}$ with the uniform distributions. Then we have
\begin{align}\label{Khinchine's inequality for functions}
\big( \mathbb E \| \sum_{j=1}^N \epsilon_j f_j \|_{L_h^p}^p \big)^\frac1p
\sim_{p} \big\| (\sum_{j=1}^N |f_j|^2)^\frac12 \big\|_{L_h^p} 
\end{align}
\end{lem}
\begin{proof}
For $x\in \Z_h^d$ we apply \eqref{Khinchine's inequality} to the sequences $f_1(x),\cdots,f_N(x)$ and then take $L_h^p$ norms
\begin{align*}
\| \big( \mathbb E|\sum_{j=1}^N \epsilon_{j}f_{j}|^{p} \big)^{\frac1p} \|_{L_h^p}^p
= \mathbb E \| \sum_{j=1}^N \epsilon_j f_j \|_{L_h^p}^p.
\end{align*} 
\end{proof}

\begin{proof}[Proof of Theorem~\ref{thm:LP}]
We first prove the second inequality in \eqref{ineq:LP}. By monotone convergence, it suffices to prove it assuming that the summation runs over finitely many $N$. That is, we suffices to show that for fixed $M\ll 1$,
\begin{align*}
\Big\| \Big( \sum_{M\le N \le 1} |P_N f|^2 \Big)^\frac12 \Big\|_{L_h^p} \le C_p \| f\|_{L_h^p}.
\end{align*}
A key observation is that for arbitrary $\epsilon_N\in\{ -1,1\}$ the multiplier $\sum_{K\le N\le1} \epsilon_N \psi_N$ obeys the assumption \eqref{Hormander condition} in the Hormander-Mikhlin Theorem. Thus we have
\begin{align*}
\| \sum_{M\le N\le1} \epsilon_N P_N f \|_{L_h^p}^p \le C_p \|f\|_{L_h^p}^p.
\end{align*}
Taking expectations on both sides and applying Khinchine's inequality \eqref{Khinchine's inequality for functions} we get the desired result.

Now we prove the first inequality in \eqref{ineq:LP}.
Note that \begin{align*}
	\|\big(\sum_{N } |P_N f |^2\big)^\frac12\|_{L_h^2} = \| f\|_{L_h^2}.
	\end{align*} Plugging $f+g$ into above identity we obtain
	\begin{align*}
	\| \sum_{N}P_N f \cdot \overline{P_N g} \|_{L_h^1} 
	= \la f, g \ra_h.
	\end{align*}
Then we have by duality
	\begin{align*}
	\|f\|_{L_h^p} &=\sup_{\| g \|_{L_h^q}\le 1} \la f, g \ra_h 
		=\sup_{\| g \|_{L_h^q}\le 1}  \| \sum_{N}P_N f \cdot \overline{P_N g} \|_{L_h^1}  \\
	&\le \sup_{\| g \|_{L_h^q}\le 1} 
	\| (\sum_{N} |P_N f|^2)^\frac12 \|_{L_h^p} \|(\sum_{N} |P_N g|^2)^\frac12 \|_{L_h^q} \\
	& \le \sup_{\| g \|_{L_h^q}\le 1}  \| (\sum_{N} |P_N f|^2)^\frac12 \|_{L_h^p}  \| g\|_{L_h^q } 
	\le C_p \| (\sum_{N} |P_N f|^2)^\frac12 \|_{L_h^p}.
	\end{align*}
\end{proof}

\subsection{Sobolev spaces and and norm equivalence (Proof of Proposition \ref{norm equivalence})}
By the definitions, differential operators $|\nabla|^s=|\nabla_h|^s$, $(-\Delta_h)^{\frac{s}{2}}$ and $D_{j;h}^+$ are Fourier multiplier operators of symbols $|\xi|^s$, $(\frac{4}{h^2} \sum_{j=1}^d \sin^2(\frac{h\xi_j}{2}))^{\frac{s}{2}}$ and $\frac{e^{ih\xi_j}-1}{h}$, respectively. Thus, applying the H\"ormander-Mikhlin theorem to the symbols, we prove the norm equivalence among Sobolev norms.

\begin{proof}[Proof of Proposition \ref{norm equivalence}]
By direct calculations, one can show that the symbols $\frac{\left(\frac{4}{h^2} \sum_{j=1}^d \sin^2(\frac{h\xi_j}{2})\right)^{\frac{s}{2}}}{|\xi|^s}$, $\frac{|\xi|^s}{\left(\frac{4}{h^2} \sum_{j=1}^d \sin^2(\frac{h\xi_j}{2})\right)^{\frac{s}{2}}}$, $\frac{\frac{e^{ih\xi_j}-1}{h}}{|\xi|}$
satisfy \eqref{Hormander condition}.  Therefore, it follows from Theorem \ref{thm: Hormander Mikhlin} that for all $1<p<\infty$, $\|(-\Delta_h)^{\frac{s}{2}}f\|_{L_h^p}\sim \|f\|_{\dot{W}_h^{s,p}}$ and $\sum_{j=1}^d\|D_{j;h}^+f\|_{L^p}\lesssim \|f\|_{\dot{W}_h^{1,p}}$.

We introduce a partition of unity $\{\chi_j\}_{j=1}^{d}$ on the unit sphere $\mathbb{S}^{d-1}$ such that $\chi_j(\xi)\equiv 1$ on $\{|\xi_j|\geq 3|\xi|\}$ but $\chi_j(\xi)\equiv 0$ on $\{|\xi_j|\leq \frac{1}{3}|\xi|\}$, and define the projection operator $\Gamma_j$ by $\widehat{\Gamma_j f}(\xi)=\chi_j(\xi) \hat{f}(\xi)$. By direct calculations again, one can show that $\chi_j(\xi)|\xi|(\frac{e^{ih\xi_j}-1}{h})^{-1}$ satisfies \eqref{Hormander condition}. As a result, we obtain $\|f\|_{\dot{W}_h^{1,p}}\le \sum_{j=1}^d \|\Gamma_j f\| _{\dot{W}_h^{1,p}}\lesssim \sum_{j=1}^d\|D_{j;h}^+f\|_{L^p}$ for all $1<p<\infty$.
\end{proof}

By the same way, one can show norm equivalence for inhomogeneous Sobolev norms (see \eqref{Sobolev norm}).

\begin{prop}
	For any $1<p<\infty$, we have
	$$\|f\|_{W_h^{s,p}}\sim \|(1-\Delta_h)^{\frac{s}{2}}f\|_{L_h^p}\quad\forall s\in\mathbb{R}$$
	and
	$$\|f\|_{W_h^{1,p}}\sim \|f\|_{L_h^p}+\sum_{j=1}^d\|D_{j;h}^+f\|_{L_h^p},$$
\end{prop}

We can also prove the relation between homogeneous and  inhomogeneous norm by the same argument.
\begin{cor}
For $1<p<\infty$ and $s\ge0$ we have 
\begin{equation}\label{ineq:bound by inhomogeneous}
\|f\|_{L_h^p} \ls \|f\|_{W_h^{s,p}} \text{ and }
\|f\|_{\dot W_h^{s,p}} \ls \|f\|_{W_h^{s,p}}.
\end{equation}
\end{cor}

\subsection{Endpoint Sobolev inequality}
We close this section deriving the endpoint Sobolev inequality, which improves Proposition \ref{Ineq: Sobolev infinity}, by the Littlewood-Paley inequality.

\begin{prop}[Endpoint Sobolev inequality]
	Let $h\in(0,1]$. Suppose that $1<p<q<\infty$ and $\frac{1}{q}=\frac{1}{p}-\frac{s}{d}$. Then, we have
	\begin{align*}
	\|f\|_{L_h^{q}} \ls \|f\|_{\dot{W}_h^{s,p}}.
	\end{align*}
\end{prop}

\begin{proof}
We first consider the case $p=2$.
By the Littlewood-Paley inequality (Theorem ~\ref{thm:LP}) and Bernstein's inequality (Lemma \ref{lem:Bernstein}), we get
\begin{align*}
\|f\|_{L_h^{q}} &\sim \left\| \left\{\sum_{N\leq 1} |P_N f|^2\right\}^{1/2} \right\|_{L_h^q}
\ls \left\{\sum_{N\leq 1} \|P_N f\|_{L_h^q}^2\right\}^{1/2}\\
&\ls \left\{\sum_{N\leq 1} \left(\frac{N}{h}\right)^s\|P_N f\|_{L_h^2}^2\right\}^{1/2}\sim \left\{\sum_{N\leq 1} \|P_N(|\nabla|^sf)\|_{L_h^2}^2\right\}^{1/2}\\
&\sim \||\nabla|^sf\|_{L_h^2}.
\end{align*}
Then the case $q=2$ follows from the standard duality arguments and Parseval's identity.

If $q>2>p$, by above two cases we have
$\| f\|_{L_h^{q}} \ls \| f\|_{\dot W_h^{s,p}}.$
By interpolating this with trivial estimate $ \| f\|_{L_h^p} \ls \| f\|_{L_h^p}$, we get the desired result.
\end{proof}

\section{Strichartz Estimates for Discrete Schr\"odinger Equations (Proof of Theorem \ref{thm: Strichartz for Schrodinger})}

In this section, we show Strichartz estimates for discrete Schr\"odinger equations. Now that harmonic analysis tools are at hand, their proof is reduced to the proof of the following frequency localized estimates.

\begin{prop}[Frequency localized dispersive estimate for discrete Schr\"odinger equations]\label{dispersive estimate for schrodinger}
Let $h\in(0,1]$. Then, for any dyadic number $N\in 2^{\mathbb{Z}}$ with $N\leq 1$, we have
\begin{equation}\label{eq: dispersive estimate for schrodinger}
\|e^{it\Delta_h} P_N u_0\|_{L_h^\infty} \lesssim \left(\frac{N}{h|t|}\right)^{d/3}\|u_0\|_{L_h^1}.
\end{equation}
\end{prop}

\begin{proof}[Proof of Theorem \ref{thm: Strichartz for Schrodinger}, assuming Proposition \ref{dispersive estimate for schrodinger}]
By Proposition \ref{dispersive estimate for schrodinger} and the trivial inequality
$$\|e^{it\Delta_h} P_N u_0\|_{L_h^2}=\|P_N u_0\|_{L_h^2}\leq\|u_0\|_{L_h^2},$$
it follows from Keel-Tao \cite{KT98} that the frequency localized Strichartz estimate
$$\| e^{it\Delta_h}P_N u_0\|_{L_t^q(\mathbb{R};L_h^r)}\ls \left(\frac{N}{h}\right)^{\frac{1}{q}}\|u_0\|_{L_h^2}$$
holds for all admissible pairs $(q,r)$ (see \eqref{S-admissible}). Let $\tilde{\psi}$ be a smooth function such that $\tilde{\psi}\equiv 1$ on $\supp \psi$, and define the operator $\tilde{P}_N$ as the Fourier multiplier of symbol $\tilde{\psi}(\frac{h}{N}\cdot)$. Then, by the Littlewood-Paley inequality and the Minkowskii inequality with $q,r\geq 2$, we prove that
\begin{align*}
\|e^{it\Delta_h}u_0\|_{L_t^q(\mathbb{R};L_h^r)}&\sim \left\|\left\{\sum_{N\leq 1} |e^{it\Delta_h}P_Nu_0|^2\right\}^{1/2}\right\|_{L_t^q(\mathbb{R};L_h^r)}\\
&\leq \left\{\sum_{N\leq 1} \|e^{it\Delta_h}P_Nu_0\|_{L_t^q(\mathbb{R};L_h^r)}^2\right\}^{1/2}=\left\{\sum_{N\leq 1} \|e^{it\Delta_h}P_N\tilde{P}_Nu_0\|_{L_t^q(\mathbb{R};L_h^r)}^2\right\}^{1/2}\\
&\lesssim\left\{\sum_{N\leq 1}\left(\frac{N}{h}\right)^{\frac{2}{q}} \|\tilde{P}_Nu_0\|_{L_h^2}^2\right\}^{1/2}\sim\left\{\sum_{N\leq 1} \|\tilde{P}_N(|\nabla|^{\frac{1}{q}}u_0)\|_{L_h^2}^2\right\}^{1/2}\\
&\sim \||\nabla|^{\frac{1}{q}}u_0 \|_{L_h^2}.
\end{align*}
The second inequality in the theorem can be proved by the same way.
\end{proof}

\begin{proof}[Proof of Proposition~\ref{dispersive estimate for schrodinger}]
By the Fourier transform, a solution to a discrete Schr\"odinger equation is represented as 
\begin{equation}\label{eq:propagator}
\begin{aligned}
e^{it\Delta_h} P_N u_0 (x)
 &=\frac{1}{(2\pi)^d}\int_{\tph} e^{-\frac{4it}{h^2} \sum_{j=1}^d \sin^2 \frac{h\xi_j}{2}}\psi(\tfrac{h\xi}{N})\hat{u}_0(\xi) e^{ix \cdot\xi}d\xi \\
 &=h^d\sum_{y\in\Z_h^d} u_0(y) \left\{\frac{1}{(2\pi)^d} \int_{\tph} e^{-i(\frac{4t}{h^2}\sum_{i=1}^d  \sin^2 \frac{h\xi_j}{2}-(x-y) \cdot \xi)}\psi(\tfrac{h\xi}{N})d\xi\right\}\\
&=(I_{N,t}*u_0)(x)
\end{aligned}
\end{equation}
for all $x\in\mathbb{Z}_h^d$, where
\begin{equation}\label{I_N}
I_{N,t}(x):=\frac{1}{(2\pi)^d} \int_{\tph} e^{-i\varphi_t(\xi)}\psi(\tfrac{h\xi}{N}) d\xi,
\end{equation}
where
$$\varphi_t(\xi):=\frac{4t}{h^2}\sum_{i=1}^d \sin^2 \frac{h\xi_j}{2}-x \cdot \xi=\frac{2t}{h^2}\sum_{i=1}^d (1-\cos h\xi_j)-x \cdot \xi.$$
We observe that 
$$\partial_{\xi_j}\partial_{\xi_k}\varphi_t(\xi)=2t\cos h\xi_j\cdot\delta_{jk},$$
and so the Hessian $H\varphi_t$ is degenerate if and only if $\xi_j=\pm\frac{\pi}{2h}$ for some $j$.

Suppose that $N=1$ or $\frac{1}{2}$ or $\frac14$. Then, by scaling $h\xi\mapsto\xi$, we have 
\begin{align*}
I_{N,t}(hx)&=\frac{1}{(2\pi)^d} \int_{\tph} e^{-i(\frac{2t}{h^2}\sum_{i=1}^d  (1-\cos h\xi_j)-hx \cdot \xi)}\psi_{N}(\xi) d\xi\\
&=\frac{1}{(2\pi h)^d} \int_{\mathbb{T}^d} e^{-i(\frac{2t}{h^2}\sum_{i=1}^d  (1-\cos \xi_j)-x \cdot \xi)}\psi(\tfrac{\xi}{N}) d\xi\\
&=\frac{1}{h^d} \left(e^{i\frac{t}{h^2}\Delta_1}(\mathcal{F}_1^{-1}\psi(\tfrac{\cdot}{N}))\right)(x)
\end{align*}
for any $x\in\mathbb{Z}^d$, where $\mathcal{F}_1^{-1}$ is the inverse Fourier transform on $\mathbb{Z}^d$. Hence, it follows from the dispersion estimate on $\mathbb{Z}^d$ (see \eqref{SK dispersion estimate}) that for all $x\in\mathbb{Z}_h^d$, 
$$|I_{N,t}(x)|\lesssim\frac{1}{h^d}\cdot\left(\frac{h^2}{|t|}\right)^{d/3}\|\mathcal{F}_1^{-1}\psi(\tfrac{\cdot}{N})\|_{L_h^1}\lesssim\frac{1}{(h|t|)^{d/3}}.$$
Therefore, going back to \eqref{eq:propagator}, we obtain \eqref{eq: dispersive estimate for schrodinger}.

If $N\leq\frac{1}{8}$, then on the support of $\psi(\tfrac{h}{N}\cdot)$, the Hessian of the phase function is non-degenerate and moreover
\begin{equation}\label{Hessian lower bound}
|\textup{det}H\varphi_t|=\left|\prod_{j=1}^d 2t\cos h\xi_j\right|\gtrsim |t|^{d}.
\end{equation}
Therefore, it follows from the standard oscillatory integral estimate that $|I_{N,t}(x)|\lesssim |t|^{-d/2}$ and 
$$\|e^{it\Delta_h} P_N u_0\|_{L_h^\infty} \lesssim \frac{1}{|t|^{d/2}}\|u_0\|_{L_h^1}\quad\underset{\textup{interpolation}}\Longrightarrow\quad\|e^{it\Delta_h} P_N u_0\|_{L_h^r} \lesssim \frac{1}{|t|^{d(\frac{1}{2}-\frac{1}{r})}}\|u_0\|_{L_h^{r'}}\quad r\geq 2.$$
Finally, inserting $\tilde{P}_N$ defined in the proof of Theorem \ref{thm: Strichartz for Schrodinger} and then employing Bernstein's inequality, we prove that
\begin{align*}
\|e^{it\Delta_h} P_N u_0\|_{L_h^\infty}&=\|\tilde{P}_N(e^{it\Delta_h} P_N u_0)\|_{L_h^\infty}\\
&\lesssim \left(\frac{N}{h}\right)^{\frac{d}{6}}\|e^{it\Delta_h} P_N u_0\|_{L_h^6}= \left(\frac{N}{h}\right)^{\frac{d}{6}}\|e^{it\Delta_h} P_N\tilde{P}_N u_0\|_{L_h^6}\\
&\lesssim \left(\frac{N}{h}\right)^{\frac{d}{6}}\cdot\frac{1}{|t|^{d/3}}\|\tilde{P}_N u_0\|_{L_h^{6/5}}\lesssim \left(\frac{N}{h}\right)^{\frac{d}{3}}\cdot\frac{1}{|t|^{d/3}}\|u_0\|_{L_h^1}.
\end{align*}
\end{proof}

\section{Uniform boundedness for Discrete Nonlinear Schr\"odinger Equations (Proof of Theorem \ref{thm: uniform bound})}

In this section, we give a simple proof of global well-posedness for the discrete nonlinear Schr\"odinger equation \eqref{eq: discrete NLS}, and then establish improved uniform boundedness of solutions employing Strichartz estimates and harmonic analysis tools developed in the previous sections.

\begin{prop}[Global well-posedness]\label{GWP}
Let $h>0$ and $p>1$. For any initial data $u_{h,0}\in L_h^2$, there exists a unique global strong solution to the discrete nonlinear Schr\"odinger equation \eqref{eq: discrete NLS},
\begin{equation}\label{Duhamel}
u_h(t)=e^{it\Delta_h}u_{h,0}- i\lambda\int_0^t e^{i(t-t_1)\Delta_h}(|u_h|^{p-1}u_h)(t_1)dt_1\in C_t(\mathbb{R}; L_h^2).
\end{equation}
Moreover, it conserves the mass \eqref{Conservation:Mass} and the energy \eqref{Conservation:Energy}.
\end{prop}

\begin{proof}
We prove local well-posedness by a standard contraction mapping argument and the trivial embedding $L_h^2\hookrightarrow L_h^\infty$, that is, nothing but $\ell^2\hookrightarrow \ell^\infty$ for sequences however whose implicit constant depends on $h>0$.

Let $I=[-T,T]$ with small $T>0$ to be chosen later. We define the nonlinear mapping
$$\Phi(u)=\Phi_{u_{h,0}}({ u }):=e^{it\Delta_h}u_{h,0}-i\lambda\int_0^t e^{i(t-t_1)\Delta_h}(|u|^{p-1}u)(t_1)dt_1.$$
Then, by unitarity of the linear propagator $e^{it\Delta_h}$ and the inequality $\|u\|_{L_h^\infty}\leq C_h\|u\|_{L_h^2}$, we get
\begin{align*}
\|\Phi(u)\|_{C(I;L_h^2)}&\leq\|u_{h,0}\|_{L_h^2}+\lambda\||u|^{p-1}u\|_{L_t^1(I; L_h^2)}\\
&\leq \|u_{h,0}\|_{L_h^2}+2\lambda T\|u\|_{C(I;L_h^\infty)}^{p-1}\|u\|_{C(I;L_h^2)}\\
&\leq \|u_{h,0}\|_{L_h^2}+2\lambda C_h^{p-1}T\|u\|_{C(I;L_h^2)}^{p}.
\end{align*}
Similarly for the difference, using the fundamental theorem of calculus
\begin{equation}\label{FTC application}
\begin{aligned}
|u|^{p-1}u-|v|^{p-1}v&=\int_0^1\frac{d}{ds}\left(|su+(1-s)v|^{p-1}(su+(1-s)v)\right)ds\\
&=\frac{p+1}{2}\int_0^1|su+(1-s)v|^{p-1}ds\cdot(u-v)\\
&\quad+\frac{p-1}{2}\int_0^1|su+(1-s)v|^{p-3}(su+(1-s)v)^2ds\cdot\overline{u-v},
\end{aligned}
\end{equation}
we show that 
\begin{align*}
\|\Phi(u)-\Phi(v)\|_{C(I;L_h^2)}&\leq \lambda\||u|^{p-1}u-|v|^{p-1}v\|_{L_t^1(I; L_h^2)}\\
&\leq 2\lambda p T\int_0^1\|su+(1-s)v\|_{C(I;L_h^\infty)}^{p-1}ds\cdot \|u-v\|_{C(I;L_h^2)}\\
&\leq 2\lambda pC_h^{p-1}T\left(\|u\|_{C(I;L_h^2)}+\|v\|_{C(I;L_h^2)}\right)^{p-1}\|u-v\|_{C(I;L_h^2)}.
\end{align*}
Let $R\geq 2\|u_{h,0}\|_{L_h^2}$. Then, taking small $T>0$ depending on $R$ and $C_h$, we prove that $\Phi$ is contractive on a ball of radius $R$ centered at zero in $C(I;L_h^2)$. Thus, the equation \eqref{eq: discrete NLS} has a unique strong solution, denoted by $u_h(t)$.

The conservation laws can be proved as usual by differentiating the mass and the energy, substituting $\partial_t u_h$ by the equation and then doing summation by parts. Note that unlike the Euclidean domain, the Laplacian $\Delta_h$ is bounded on $L_h^2$, and thus the energy is properly defined for $L_h^2$-solutions.

The mass conservation prevents a solution to blow up in $L_h^2$ in finite time. Therefore, $u_h(t)$ exists globally in time.
\end{proof}

Next, we will show the improved uniform boundedness (Theorem \ref{thm: uniform bound}). To this end, we need the following nonlinear estimate.

\begin{lem}
Suppose $p>1$. Then,
\begin{equation}\label{nonlinear estimate 1}
\||u_h|^{p-1}u_h\|_{L_t^1(I;H_h^1)}\lesssim \|u_h\|_{L_t^{p-1}(I;L_h^\infty)}^{p-1}\|u_h\|_{C(I;H_h^1)}.
\end{equation}
\end{lem}
\begin{proof}
By the norm equivalence (Proposition \ref{norm equivalence}), we write 
$$\||u_h|^{p-1}u_h\|_{L_t^1(I;H_h^1)}\sim \||u_h|^{p-1}u_h\|_{L_t^1(I;L_h^2)}+\sum_{j=1}^d\|D_{j;h}^+(|u_h|^{p-1}u_h)\|_{L_t^1(I;L_h^2)}.$$
Then, applying the fundamental theorem of calculus to 
$$D_{j;h}^+(|u|^{p-1}u)=\frac{(|u|^{p-1}u)(x+he_j)-(|u|^{p-1}u)(x)}{h}$$
as in \eqref{FTC application}, one can bound $\|D_{j;h}^+(|u_h|^{p-1}u_h)\|_{L_t^1(I;L_h^2)}$ by 
\begin{align*}
&p \int_0^1\|su_h(x+he_j)+(1-s)u_h(x)\|_{L_t^{p-1}(I;L_h^\infty)}^{p-1}ds\cdot \left\|\frac{u_h(x+he_j)-u_h(x)}{h}\right\|_{C(I;L_h^2)}\\
&\leq p\left(\|u_h(\cdot+he_j)\|_{L_t^{p-1}(I;L_h^\infty)}+\|u_h\|_{L_t^{p-1}(I;L_h^\infty)}\right)^{p-1}\|D_{j;h}^+u_h\|_{C(I;L_h^2)}\\
&= p2^{p-1} \|u_h\|_{L_t^{p-1}(I;L_h^\infty)}^{p-1}\|D_{j;h}^+u_h\|_{C(I;L_h^2)}.
\end{align*}
Therefore, by the norm equivalence again, we obtain 
\begin{equation*}
\||u_h|^{p-1}u_h\|_{L_t^1(I;H_h^1)}\lesssim \|u_h\|_{L_t^{p-1}(I;L_h^\infty)}^{p-1}\|u_h\|_{C(I;H_h^1)},
\end{equation*}
where the implicit constant is independent of $h\in(0,1]$.
\end{proof}

\begin{proof}[Proof of Theorem \ref{thm: uniform bound}]
Let $I=[-\tau,\tau]$ be a sufficiently small interval. We apply Strichartz estimates (Theorem \ref{thm: Strichartz for Schrodinger}) to the solution \eqref{Duhamel} to get 
\begin{equation}\label{uniform bound proof}
\|u_h\|_{S^1(I)}\leq C\|u_{h,0}\|_{H_h^1}+C\lambda\||u_h|^{p-1}u_h\|_{L_t^1(I;H_h^1)},
\end{equation}
where $C>0$ is a uniform constant.

Next, we claim that there is $\alpha>0$ such that 
\begin{equation}\label{improved uniform bound claim}
\|u_h\|_{L_t^{p-1}(I;L_h^\infty)}\lesssim \tau^\alpha \|u_h\|_{S^1(I)}.
\end{equation}
Indeed, if $d=2,3$, then by the assumption $p<1+\frac{4}{d-2}$, there exists small $\delta>0$ such that $\alpha:=\frac{1}{p-1}-\frac{d-2(1-\delta)}{4}>0$. Hence, applying the H\"older inequality and the Sobolev inequality, and then using that $(\frac{4}{d-2(1-\delta))},\frac{4d}{6(1-\delta)-d})$ is admissible, we get
\begin{equation}\label{not high d}
\begin{aligned}
\|u_h\|_{L_t^{p-1}(I;L_h^\infty)}&\leq (2\tau)^\alpha\|u_h\|_{L_t^{\frac{4}{d-2(1-\delta)}}(I; L_h^\infty)}\\
&\lesssim \tau^\alpha\|u_h\|_{L_t^{\frac{4}{d-2(1-\delta)}}(I; W_h^{\frac{6-d-2\delta}{4}, \frac{4d}{6(1-\delta)-d}})}\\
&\leq \tau^\alpha \|u_h\|_{S^1(I)}.
\end{aligned}
\end{equation}
By the same way but with the one-dimensional Sobolev inequality, one can prove the claim \eqref{improved uniform bound claim}.

Inserting the bound \eqref{improved uniform bound claim} in \eqref{nonlinear estimate 1}, we get
$$\||u_h|^{p-1}u_h\|_{L_t^1(I;H_h^1)}\lesssim \tau^{\alpha(p-1)} \|u_h\|_{S^1(I)}^p,$$
and going back to \eqref{uniform bound proof}, we obtain that 
$$\|u_h\|_{S^1(I)}\leq C \|u_{h,0}\|_{H_h^1}+\tilde{C}\tau^{\alpha(p-1)}\|u_h\|_{S^1(I)}^p.$$
Therefore, we may increase $\tau$ up to $\frac{1}{(2C\|u_{h,0}\|_{H_h^1})^{1/\alpha}}\frac{1}{(2\tilde{C})^{1/\alpha(p-1)}}$, keeping the bound
\begin{equation}\label{uniform bound on I}
\|u_h\|_{S^1(I)}\leq 2C \|u_{h,0}\|_{H_h^1}.
\end{equation}

It remains to show global-in-time bound $(ii)$. If $\lambda>0$ and $\frac{1}{p}>\max\{\frac{d-2}{d+2},0\}$, then by the energy conservation laws, any solution $u_h(t)$ satisfies 
$$\frac{1}{2}\|u_h(t)\|_{\dot{H}_h^1}^2\leq E_h(u_h(t))=E_h(u_{h,0}).$$
On the other hand, if $\lambda<0$ and $p<1+\frac{4}{d}$, then it follows from the Gagliardo-Nirenberg inequality and the mass and the energy conservation laws that
	\begin{align*}
	E_h(u_{h,0})&=E_h(u_h(t))\\
	&\ge \frac12 \|u_h(t) \|_{\dot{H}_h^1}^2 -C\left(\| u_{h}(t)\|_{L_h^2}^{1-\frac{d(p-1)}{2(p+1)}} \| u_h(t) \|_{\dot{H}_h^1}^{\frac{d(p-1)}{2(p+1)}}\right)^{p+1} \\
	&\ge \frac12 \|u_h(t) \|_{\dot{H}_h^1}^2 -C\| u_{h,0}\|_{L_h^2}^{p+1-\frac{d(p-1)}{2}} \| u_h(t) \|_{\dot{H}_h^1}^{\frac{d(p-1)}{2}}.
	\end{align*}
Since $\frac{d(p-1)}{2}<2$, it proves that $\|u_h(t)\|_{\dot{H}_h^1}$ is bounded uniformly in $h\in(0,1]$. Therefore, in both cases, the a priori bound allows to iterate \eqref{uniform bound on I} with the uniform size of intervals so that $I_{max}=\mathbb{R}$.
\end{proof}

\begin{rem}
In our analysis, high dimensions $d\geq 4$ are not included due to lack of admissible pairs. In \eqref{not high d}, the admissible pair $(\frac{4}{d-2(1-\delta))},\frac{4d}{6(1-\delta)-d})$ is employed, however $\frac{4}{d-2(1-\delta))}$ is less than 2 when $d\geq 4$.
\end{rem}

\section{Strichartz Estimates for Discrete Klein-Gordon Equations \\ (Proof of Theorem \ref{thm: Strichartz for KG})}
For discrete Klein-Gordon equation, it behaves like Schr\"odinger equation near the origin in the fourier side and we indeed get the same derivative loss as Schr\"odinger case. But in the large frequency region it is much like a wave equation. Thus higher regularity loss is required to compensate weak dispersion. As shown before, the Strichartz estimates follows from the dispersive estimates. 
\begin{prop}
Let $h\in(0,1]$. Then for any dyadic number $N\in2^{\Z}$ with $N\le1$, we have
\begin{equation}\label{dispersive:Klein}
\| e^{it\sqrt{1-\Delta_h}}P_N u_0\|_{L_h^\infty} \le C t^{-\frac13} (\frac Nh)^{\frac13} \big( 1+ \frac Nh \big)
\|u_0\|_{L_h^1}.
\end{equation}
\end{prop}
\begin{proof}
By the Fourier transform, a solution to a discrete Klein-Gordon equation is represented as 
\begin{equation}\label{eq}
e^{it\sqrt{1-\Delta_h}} P_N u_0 (x)=(I_{N,t}*u_0)(x)
\end{equation}
for all $x\in\mathbb{Z}_h^d$, where
\begin{equation*}
I_{N,t}(x):=\frac{1}{(2\pi)^d} \int_{\mathbb T_h} e^{-i\varphi_t(\xi)}\psi(\tfrac{h\xi}{N}) d\xi,
\end{equation*}
where
$$\varphi_t(\xi):=t\sqrt{1+\frac{4}{h^2}\sin^2 \frac{h\xi}{2}}-x\xi.$$	
We observe that
\begin{align*}
\varphi_t''(\xi) = \frac{th^{-2}}{\sqrt{1+\frac{4}{h^2}\sin^2(\frac{h\xi}{2})}^3}
\Big(-\cos^2{(h\xi)} +(h^2+2)\cos{(h\xi)}-1 \Big),
\end{align*}	
and so the second order derivative is degenerate if and only if $\cos{(h\xi)}=1-\frac{h}{h+\sqrt{h^2+2}}$. We denote this degenerate point by $\xi_h$. By Taylor's expansion, we have 
$$\cos h > 1-\frac{h^2}{2}> 1-\frac{h}{h+\sqrt{h^2+2}}=\cos{(h\xi_h)},$$ which implies that $1<|\xi_h| $.
	
We first consider the case $\frac{4N}{h}\le1$. In this support we have $|\xi|<1$, where the degenerate point $\xi_h$ is excluded so the lower bound on the second derivative can be obtained.
Since $|\xi|\le 1$ we estimate using Taylor expansion
\begin{align*}
-\cos^2{(h\xi)} +(h^2+2)\cos {(h\xi)}-1
&\ge -\cos^2 h +(h^2+2)\cos h-1  \\
&> 1-\cos^2h-\frac{h^4}{2}=h^2+O(h^4).
\end{align*}
And it holds $ \sqrt{1+\frac{4}{h^2}\sin^2(\frac{h\xi}{2})} \ls 1$.
From these two inequalities we obtain
\begin{equation}\label{lb}
|\varphi''(\xi)|\gtrsim |t|,
\end{equation}
which is the same bound as the Schr\"odinger case \eqref{Hessian lower bound} as expected. Then by the same argument below \eqref{Hessian lower bound} we obtain
$$
\| e^{it\sqrt{1-\Delta_h}}P_N u_0\|_{L_h^\infty} \le C t^{-\frac13} (\frac Nh)^{\frac13} 
\|u_0\|_{L_h^1},
$$
which implies \eqref{dispersive:Klein}.

Next we consider the case $\frac{4N}{h}\ge1$, where it holds $|\xi|\gtrsim 1$.
So, in this support the degenerate point $\xi_h$ can be contained. We find the third derivative of $\varphi_t$ 
\begin{align*}
\varphi_t^{(3)}(\xi)= \frac{-th^{-3}\sin{(h\xi)}}{\sqrt{1+\frac{4}{h^2}\sin^2(\frac{h\xi}{2})}^5}
\Big(\cos^2{(h\xi)} -(h^2+2)\cos{(h\xi)} + h^4+4h^2+1 \Big).
\end{align*}
Note that for all $\xi\in \mathbb{T}_h$ it holds
\begin{align*}
& \cos^2{(h\xi)} -(h^2+2)\cos{(h\xi)} + h^4+4h^2+1 \gtrsim h^2, \\
& \sqrt{1+\frac{4}{h^2}\sin^2(\frac{h\xi}{2})} \ls |\xi|,
\end{align*}
where in the first line it attains minimum value $3h^2+h^4$ at $\cos{(h\xi)}=1$.
Using these, we find the low bound on the third derivative 
\begin{align*}
| \varphi_t^{(3)}(\xi) | \gtrsim |t| |\xi|^{-4},
\end{align*}
which implies by the Van der Corput Lemma,
\begin{align*}
|I_{N,t}(x)| \ls t^{-\frac13} \big( \frac Nh \big)^{\frac43}.
\end{align*}
From this we obtain \eqref{dispersive:Klein} by applying Young's convolution inequality to \eqref{eq}.
\end{proof}

\appendix

\section{Appendix}
In this appendix we consider the optimality of Theorem~\ref{thm: Strichartz for Schrodinger}.
The next proposition says that $(q,r)$ range and loss of derivative in Theorem~\ref{thm: Strichartz for Schrodinger}  can not be improved. 
We prove sharpness by adapting the standard `Knapp' example. 
We modify the example to make it applicable in our setting, i.e., to be defined on $\Z_h^d$. And then we compute its norm and observe how it depends on $h>0$ when $h$ goes to zero.
Our proof follows the mainstream of \cite[Proposition~1]{SK2005}.
\begin{prop}\label{Prop:optimal}
	Suppose for some $s\in \R$ and $2\le p,q\le \infty$ we have
	\begin{align}\label{strichartz2}
	\| e^{it\Delta_h} f \|_{L_t^q L_h^r}
	&\ls \| f \|_{\dot{H}_h^{s}}.
	\end{align} Then it should hold $\frac d2\ge \frac3q+\frac dr$ and $s\ge \frac1q$.	
\end{prop}
\begin{proof}
	By duality argument, \eqref{strichartz2} is equivalent to
	\begin{align}\label{dual}
	\Big\| \int e^{it\Delta_h} |\nabla_h|^{-s}f(t) dt \Big\|_{L_h^2}
	\ls \|f\|_{L_t^{q'}L_h^{r'}},
	\end{align}
	where $(q',r')$ is the H\"older conjugate of $(q,r)$.
	Applying Plancherel's theorem we compute the left side
	\begin{align*}
	\Big\| \int e^{it\Delta_h}|\nabla_h|^{-s}f (t)dt\Big\|_{L_h^2}
	&= \Big\| \int e^{it\frac{4}{h^2}\sum_{i=1}^d\sin^2(\frac{h\xi_i}{2})} |\xi|^{-s}\widehat{f}(t,\xi)  dt \Big\|_{L_\xi^2(\tph)} \\
	&=\Big\| |\xi|^{-s} \mathcal{F}_{t,x} f
	\big(-\sum_{i=1}^d\frac{4}{h^2}\sin^2(\frac{h\xi_i}{2}) ,\xi\big)  \Big\|_{L_\xi^2(\tph)}.
	\end{align*}
	 Now we choose $f$ as the 'Knapp' example, that is, 
	\begin{align*}
	\mathcal{F}_{t,x}f(\tau,\xi):=
	\mathbf{1}_{(-1,1)}\bigg( \epsilon^{-3}h^2 \Big(\tau+\frac{d(2-\pi)}{h^2}+\frac{2}{h}\sum_{i=1}^d\xi_i\Big)\bigg)
	\prod_{i=1}^{d}\mathbf{1}_{(-1,1)}\bigg(\epsilon^{-1}(\frac{h\xi_i}{2}-\frac{\pi}{4})\bigg).
	\end{align*}
By Taylor expansion for $\sin^2 y$ around $y=\frac{\pi}{4}$, it holds that 
\begin{align*}
\Big| -\sum_{i=1}^d\sin^2(\frac{h\xi_i}{2})
+\Big( \frac{d}{2} + \sum_{i=1}^d(\frac{\pi}{4}-\frac{h\xi}{2}) \Big) \Big| \le C\big|\frac{\pi}{4}-\frac{h\xi}{2}\big|^3,
\end{align*}
where the constant $C>0$ is independent of $h$.
Thus if $|\frac{h\xi_i}{2}-\frac{\pi}{4}|\le \epsilon$ for $i=1,2,\cdots,d$ we have
	\begin{align*}
	\Big| -\sum_{i=1}^d\frac{4}{h^2}\sin^2(\frac{\xi_i}{2})+\frac{d(2-\pi)}{h^2}+\frac{2}{h}\sum_{i=1}^d\xi_i \Big|\le Ch^{-2}\epsilon^3,
	\end{align*}
which implies
\begin{align*}
\mathcal{F}_{t,x} f
\Big(-\sum_{i=1}^d\frac{4}{h^2}\sin^2(\frac{h\xi_i}{2}) ,\xi\Big)
\approx \prod_{i=1}^{d}\mathbf{1}_{(-1,1)}\Big(\epsilon^{-1}h(\frac{\xi_i}{2}-\frac{\pi}{4h} ) \Big).
\end{align*}	
Inserting this example gives
	\begin{align}\label{leftnorm}
	\Big\| \int e^{it\Delta_h} |\nabla_h|^{-s}f (t)dt \Big\|_{L_h^2}
	\sim h^{s}(\ep h^{-1})^\frac{d}{2}.
	\end{align}
Next we consider the right side of \eqref{dual}. We represent it in terms of $f$ using inversion formula and compute the Fourier transform of characteristic function 
	\begin{align*} 
	| f(t,x) |
	&= \Big| \frac{1}{(2\pi)^d} \int_{\tph} \int_{\R} e^{it\tau} e^{ix\cdot\xi}\mathcal{F}_{t,x}f(\tau,\xi) d\tau  d\xi \Big| \\
	&=C \Big| \frac{\sin(\ep^3h^{-2}t)}{t}\Big| \cdot \prod_{i=1}^d \Big|
	\frac{ \sin\big(\epsilon h^{-1}(x_i-\frac{2t}{h})\big) }{x_i-\frac{2t}{h}} \Big|,
	\end{align*}
whenever $\epsilon h^{-2}\le  \frac{\pi}{2}$. Here we denote $x=(x_1,\cdots,x_d)\in \Z_h^d$ and the constant $C$ is independent of $h$ and $\ep$.
Note that for $1\le p<\infty$ we have
	\begin{align*}
	h\sum_{x_i\in \Z_h} \frac{ |\sin(\delta(x_i-y))|^p }{|x_i-y|^p}
	&= h\sum_{|x_i-y|\le\delta^{-1}}\frac{ |\sin(\delta(x_i-y))|^p }{|x_i-y|^p}
	+ h\sum_{|x_i-y|\ge\delta^{-1}}\frac{ |\sin(\delta(x_i-y))|^p }{|x_i-y|^p}  \\
	&\ls h\sum_{|x_i-y|\le\delta^{-1}}\delta^p+h\sum_{|x_i-y|\ge\delta^{-1}}\frac{1}{|x_i-y|^p}
	\ls \delta^{p-1},
	\end{align*}
	where the implicit constants are independent of $h>0$ and $y\in\R$.
	Using this, we estimate 
	\begin{align}
	\begin{aligned}\label{rightnorm}
	\|f\|_{L_t^{q'}L_h^{r'}(\Z_h^d)} 
	&= \bigg\|   \frac{\sin(\ep^3h^{-2}t)}{t} \cdot \Big\| \prod_{i=1}^d \big|
	\frac{ \sin\big(\epsilon h^{-1}(x_i-\frac{2t}{h})\big) }{x_i-\frac{2t}{h}} \big|  \Big\|_{L_h^{r'}(\Z_h^d)} \bigg\|_{L_t^{q'}} \\
	&\ls (\epsilon h^{-1})^{d(1-\frac{1}{r'})} \Big\| \frac{\sin(\ep^3h^{-2}t)}{t}  \Big\|_{L_t^{q'}} \\
	&\ls (\epsilon h^{-1})^{d(1-\frac{1}{r'})} (\epsilon
	^3 h^{-2})^{(1-\frac{1}{q'})}.
	\end{aligned}\end{align}
	Thus if \eqref{dual} holds true, it should be satisfied by \eqref{leftnorm} and \eqref{rightnorm} 
	\begin{align*}
	h^{-\frac d2+s} \ep^{\frac d2} \ls h^{-2-d+\frac{d}{r'}+\frac{2}{q'}}\epsilon^{3(1-\frac{1}{q'})+d(1-\frac{1}{r'})}.
	\end{align*}
	Letting $\ep$ and $h$ go to zero with $\epsilon h^{-2}\le  \frac{\pi}{2}$, we obtain
	\begin{align*}
	\frac2d \ge 3(1-\frac{1}{q'})+d(1-\frac{1}{r'}) , \quad 
	2-\frac d2+s \ge -d+\frac{d}{r'}+\frac{2}{q'}, 
	\end{align*}
	which implies the desired result.
\end{proof}

\bibliographystyle{amsplain}
\bibliography{uniformStrichartz}

\end{document}